\documentclass{amsart}[11pt]

\usepackage[left=1.25in, right=1.25in, top=1in, bottom=1in]{geometry}
\usepackage{amsmath, amsthm, amssymb, mathtools, stmaryrd}

\usepackage{framed}
\usepackage{enumerate}
\usepackage{changepage}
\usepackage{verbatim}
\usepackage{tikz-cd}
\usepackage{lineno}
\usepackage[
backend=biber,
style=alphabetic,
maxalphanames=99,
maxbibnames=99
]{biblatex}

\usepackage[draft]{todonotes}   
\usepackage{soul} 
\usepackage{braket}
\usepackage{url}
\usepackage[symbol]{footmisc}

\usepackage{etoolbox} 
\usepackage{mathrsfs} 
\usepackage[pdfencoding=auto, psdextra]{hyperref}
\usepackage{cleveref}
\usepackage[normalem]{ulem} 

\newcommand{\CM}{\operatorname{C-M}}

\newcommand{\mc}[1]{\mathcal{#1}}

\newcommand{\ol}[1]{\overline{#1}}
\newcommand{\ovl}[1]{\overline{#1}}
\newcommand{\mscr}{\mathscr}

\newcommand{\Fal}{\operatorname{Fal}}
\newcommand{\h}{\widehat{h}}
\newcommand{\Z}{\ensuremath{\mathbb{Z}}}

\newcommand{\wh}{\widehat}

\newcommand{\bb}{\mathbb}

\newcommand{\GVF}{\operatorname{GVF}}

\newcommand{\mcL}{\ensuremath{\mathcal{L}}}

\newcommand{\mcO}{\ensuremath{\mathcal{O}}}

\newcommand{\abs}[1]{\ensuremath{\left\| #1 \right\|}}

\renewcommand{\mod}{\mathrm{mod}}

\DeclareMathOperator{\Aut}{Aut}
\DeclareMathOperator{\Stab}{Stab}

\DeclareMathOperator{\Gal}{Gal}

\DeclareMathOperator{\Spec}{Spec}

\DeclareMathOperator{\Hom}{Hom}

\DeclareMathOperator{\tr}{tr}
\DeclareMathOperator{\Tr}{Tr}

\DeclareMathOperator{\Pic}{Pic}

\DeclareMathOperator{\codim}{codim}

\DeclareMathOperator{\Div}{Div}

\DeclareMathOperator{\Frac}{Frac} 

\DeclareMathOperator{\an}{an} 
\DeclareMathOperator{\divv}{div}

\newcommand{\eps}{\varepsilon}

\usepackage[bbgreekl]{mathbbol}
\usepackage{amsfonts}
\DeclareSymbolFontAlphabet{\mathbb}{AMSb} 
\DeclareSymbolFontAlphabet{\mathbbl}{bbold}

\theoremstyle{plain}
\newtheorem{thm}{Theorem}
\newtheorem{theorem}[thm]{Theorem}
\newtheorem{lemma}[thm]{Lemma}

\newtheorem{proposition}[thm]{Proposition}

\theoremstyle{definition}
\newtheorem{defn}[thm]{Definition}
\newtheorem{rmk}[thm]{Remark}
\newtheorem{remark}[thm]{Remark}

\numberwithin{thm}{section}
\numberwithin{prob}{section}
\numberwithin{equation}{section}

\title{Uniform Height Gaps in Arbitrary Characteristics}
\author{Alice Lin, Frank Lu and Jit Wu Yap}
\date{today}

\begin{document}

\begin{abstract}
We prove uniform height gaps for subvarieties of abelian varieties in arbitrary characteristic, extending the new gap principle of Gao--Ge--K{\"u}hne. Our proof goes through the theory of adelic curves and globally valued fields, and proves the Bogomolov conjecture for an arbitrary globally valued field. This gives a different way to obtain uniform Bogomolov-type results, differing from the approaches of Dimitrov--Gao--Habegger--K{\"u}hne and Yuan. First, we prove the Bogomolov conjecture over globally valued fields for divisors. We then reduce the case of general subvarieties to the case of divisors by an induction argument.  Specializing to the case of global function fields, we obtain a new proof of the geometric Bogomolov conjecture following the strategy of Gubler and Yamaki.
\end{abstract}

\maketitle

\section{Introduction}
In their breakthrough paper, Gao--Ge--K{\"u}hne proved the following new gap principle for subvarieties $V$ of an abelian variety $A/\ovl{\bb{Q}},$ as \cite[Theorem 1.2]{GGK21}. We say that a subvariety $V$ generates $A$ if the smallest abelian variety containing $V-V$ is $A$ itself. One consequence of their theorem is the following result.

\begin{theorem}\label{IntroTheorem1}
There exists $\eps = \eps(g,d) > 0$ with the following property: for any polarized abelian variety $(A,L)$ over $\ovl{\bb{Q}}$ of dimension $g$ with $L$ symmetric, and any irreducible closed subvariety $V \subseteq A$ that generates $A$ of degree $\deg_L V \leq d$, with $\Stab(V)$ finite, we have 
$$\h_{A,L}(V) \geq \eps \max \{1, h_{\Fal}(A)\}.$$
\end{theorem}

Here, $\h_{A,L}(V)$ denotes the Neron--Tate height of the subvariety $V$ with respect to $L$. Theorem \ref{IntroTheorem1} plays an essential role in proving the uniform Mordell-Lang conjecture \cite[Theorem 1.1]{GGK21}. For curves embedded in their Jacobians, the analogue of Theorem \ref{IntroTheorem1} was shown by by Dimitrov--Gao--Habegger \cite{DGH21} and K{\"u}hne \cite{Kuh21} as part of their proof of the uniform Mordell conjecture and uniform Bogomolov conjecture for curves in Jacobians respectively. Yuan \cite{Yua24} subsequently gave an alternative proof for curves in Jacobians, that holds in any characteristic. 
\par 
Our first main result is a new gap principle in arbitrary characteristic. For this article, we say that $K$ is a global function field if $K = k(B)$, where $k$ is a field and $B$ a smooth projective curve over $k$. A global function field $K$ naturally comes with a global height function \cite{BG06}. The Faltings height $h_{\Fal}(A)$ can be analogously defined in the function field case, see \cite[Definition 1.2]{GLFP25} where they use the term differential height instead.

\begin{theorem} \label{IntroTheorem2}
There exists positive constants $\eps = \eps(g,d) > 0, M = M(g,d) > 0$ such that the following holds: let $K$ be a global function field and $(A,L)$ a polarized abelian variety with $L$ symmetric. For any integral subvariety $V \subseteq A_{\ovl{K}}$ satisfying $\deg_L(V) \leq d$, such that $V$ generates $A$ and $\Stab(V)$ is finite, we have 
$$\h_{A,L}(V) \geq \eps h_{\Fal}(B),$$
where $B$ is an abelian variety with an isogeny $\pi: A \to B$ of degree at most $M$. 
\end{theorem}

When $K$ is of characteristic zero, an isogeny does not change the Faltings height \cite[Theorem A]{GLFP25}. In positive characteristic however, an isogeny can drastically change the Faltings height, even when the degree of the isogeny remains fixed \cite[Proposition 6.3]{GLFP25}. As such, it is necessary to replace $h_{\Fal}(A)$ with $h_{\Fal}(B)$. 
\par 
The fact that we may bound the degree of $\pi: A \to B$ allows us to still obtain uniform bounds on small points, after moving the subvariety $V$ to $B$ using $\pi$ and then applying the quantitative Zhang's inequality \cite{MY25}. As an example, the following uniform bound on small points for curves is an immediate corollary of Theorem \ref{IntroTheorem2} and \cite[Theorem 1.1]{MY25}.

\begin{theorem} \label{IntroTheorem3}
In the setting of Theorem \ref{IntroTheorem2}, assume that $V = C$ is a curve. Then there exists positive constants $\eps = \eps(g, \deg_L(C)) > 0, M = M(g,\deg_L(C)) > 0, N = N(g,\deg_L(C)) > 0$ such that 
$$\# \left\{ x \in C(\ovl{K}) \mid \h_{A,L}(x) < \eps h_{\Fal}(B) \right\} \leq N,$$
where $B$ is an abelian variety with an isogeny $\pi: A \to B$ of degree at most $M$. 
\end{theorem}

For general subvarieties $V$, we expect that Theorem \ref{IntroTheorem2} should give a suitable analogue of the uniform Bogomolov conjecture \cite[Theorem 1.3]{GGK21} under the correct hypotheses. 
\par 
Our proof strategy for Theorem \ref{IntroTheorem2} differs from Gao--Ge--K{\"u}hne's and also from Yuan's. We follow a strategy first suggested by Hrushovski, by proving the Bogomolov conjecture over a general globally valued field. A globally valued field $K$ is a field with the data of a global height function on $\bb{P}^n(K)$, for all $n \geq 1$, satisfying some basic compatibility relations. The set of globally valued fields are stable under ultraproducts (see Section \ref{sec: IntroGVF}) and hence satisfy a form of compactness. These objects have been recently axiomatized by Ben--Destic--Hrushovski--Szachniewicz in \cite{GVFfoundations}. If $K$ is a countable GVF, then it is an adelic curve in the sense of Chen--Moriwaki (see \cite{CM20}). In particular, a robust theory of Arakelov geometry exists for such objects. 
\par 
Given a GVF $K$, we say that $K$ is archimedean if $h(2) \not = 0$ and non-archimedean otherwise. For a non-archimedean GVF, we let $k$ be the subfield of height zero elements. Now assume that $K$ is algebraically closed and $A/K$ is an abelian variety. The GVF structure on $K$ allows us to define global canonical heights of subvarieties of $A$. A subvariety $V$ is said to be special if $V$ is of the form
$$V = x + B + \tr(W \otimes K)$$
where $x \in A(K)$ has canonical height zero, $\tr: \Tr_{K/k}(A) \otimes K \to A$ is the trace morphism from the $K/k$-trace $\Tr_{K/k}(A)$ and $W$ a subvariety of $\Tr_{K/k}(A)$. 

\begin{theorem} \label{IntroTheorem4}
Let $K$ be an algebraically closed, countable GVF and $A/K$ an abelian variety. Let $L$ be a symmetric ample line bundle on $A$ and $V$ a subvariety of $A$. Then
$$\h_{A,L}(V) = 0 \text{ if and only if } V \text{ is special.}$$
\end{theorem}

Note that Theorem \ref{IntroTheorem4} implies the case of a general uncountable GVF as the data of $(A,L,V)$ can be defined over a countable subfield. For archimedean GVFs, Theorem \ref{IntroTheorem4} was proven by Chen and Moriwaki \cite{chen2022equidistribution} using their equidistribution theorem for adelic curves. In \cite{Hul26}, Hultberg proved Theorem \ref{IntroTheorem4} for characteristic zero GVFs and showed that it implies Theorem \ref{IntroTheorem1}. However his proof uses the non-degeneracy results of Gao--Ge--K{\"u}hne \cite{GGK21} and hence essentially gives the same proof of Theorem \ref{IntroTheorem1}, but repackaged. Our proof of Theorem \ref{IntroTheorem4} does not use non-degeneracy at all and gives us a new proof of Theorem \ref{IntroTheorem1}. Hultberg also shows, using ideas of Luo--Yu \cite{LY25} in the function/number field case, that the Bogomolov conjecture for semiabelian varieties reduces to the abelian variety case \cite[Theorem 2]{Hul26}. We thus obtain the Bogomolov conjecture for semiabelian varieties over GVFs as a corollary of Theorem \ref{IntroTheorem4}, which implies the new gap principle for semiabelian varieties \cite[Theorem 3]{Hul26} over number fields. A different proof of this new gap principle has also been given by Han, Luo and Yu \cite{HLY26}.   
\par 
As remarked in the abstract, our proof when specialized to the case of global function fields, gives a new proof of the geometric Bogomolov conjecture in arbitrary characteristic. For characteristic zero function fields, this was first proven for function fields of curves by Gao and Habegger \cite{GH19}, and then extended to function fields of varieties by Cantat--Gao--Habegger--Xie \cite{CGHX21}. Xie and Yuan subsequently resolved the full geometric Bogomolov conjecture \cite{XY22}, using earlier work of Gubler \cite{gubler2007bogomolov} and Yamaki \cite{Yam18}. 
\par 
Before going into the technical details of our proof of Theorem \ref{IntroTheorem4}, we think it is instructive to first understand how we decided that GVFs (and adelic curves) are the natural way to approach Theorems \ref{IntroTheorem1} and \ref{IntroTheorem2}. Let $B$ be a smooth projective curve over $\ovl{\mathbb{Q}}$ and let $\mathcal{A}/U$ be an abelian scheme over a Zariski open $U \subseteq B$. Let $\mathcal{V}/U$ be a family of subvarieties of $\mathcal{A}/U$. Let $K = \ovl{\mathbb{Q}}(B)$ be the function field and let $A/K, V/K$ be the generic fiber of $\mathcal{A}/U$ and $\mathcal{V}/U$. Then a uniform Bogomolov-type statement for the family $\mathcal{V}/U$ implies the geometric Bogomolov conjecture for the subvariety $V \subseteq A$ over the function field $\ovl{\mathbb{Q}}(B)$. Consequently, any proof of the uniform Bogomolov conjecture must already contain, in some form, a proof of the geometric Bogomolov conjecture.
\par 
Conversely, a proof of the geometric Bogomolov conjecture may provide the starting point for a uniform version. Both the approach of Dimitrov--Gao--Habegger--Kühne and that of Yuan can be understood in this way. The former develops the theory of non-degeneracy introduced by Gao and Habegger in their proof of the geometric Bogomolov conjecture for subvarieties \cite{GH19}, whereas the latter builds on Cinkir’s proof of the geometric Bogomolov conjecture for curves embedded in their Jacobians \cite{Cin11}.
\par 
Motivated by this perspective, we seek a proof of Theorem \ref{IntroTheorem1} inspired by the approach of Gubler and Yamaki to the geometric Bogomolov conjecture. We first recall the underlying strategy. Let $V \subseteq A$ be over $\ovl{\mathbb{Q}}$ with $\h_A(V) = 0$. In their proofs of the original Bogomolov conjecture, both Ullmo \cite{Ulm98} and Zhang \cite{Zha98} construct two unequal measures $\mu_1,\mu_2$ over $\mathbb{C}$ before using the equidistribution theorem of Szpiro--Ullmo--Zhang \cite{SUZ97} to conclude that $\mu_1 = \mu_2$. This gives a contradiction. 
\par 
In the function-field setting, Chambert--Loir \cite{CL06} constructed non-Archimedean analogues of the measures $\mu_1,\mu_2$ on the Berkovich analytification $A^{\an}_v$ for a non-archimedean place $v$. Gubler \cite{gubler2007tropical, gubler2007bogomolov, gubler2010non} studied these measures in detail and proved that they are not equal when $A$ has totally degenerate reduction at $v$, proving the geometric Bogomolov conjecture in the totally degenerate case. Yamaki \cite{yamaki2013geometric, Yam16, Yam18} later extended Gubler's approach beyond the totally degenerate setting, notably to effective divisors $V \subseteq A$ and when $A$ has bad reduction at some place $v$. 
\par 
The approach of Gubler and Yamaki ultimately relies on establishing the non-equality of two non-Archimedean measures. Statements of this kind have been used by the third author to prove uniform versions of the dynamical Bogomolov conjecture \cite{GY26} through the following strategy. One first decomposes the global canonical height as a sum of non-negative local contributions 
$$\h(V) = \sum_{v \in M_K} \wh{\lambda}_{v}(V).$$ 
Using an ultrafilter degeneration argument, following Luo \cite{Luo21, Luo22} and Favre--Gong \cite{FG24}, see also \cite{LY26}, we are able to deduce uniform lower bounds on $\wh{\lambda}_v(V)$ from a non-archimedean rigidity theorem. Roughly speaking, this rigidity statement asserts that two non-Archimedean dynamical measures, under suitable bad reduction hypothesis, can coincide only when the underlying dynamical systems are related. Summing up over $v \in M_K$ then gives us the desired uniform lower bound on $\h(V)$. 
\par 
The main obstacle to carrying out this strategy in our setting is that, for subvarieties of abelian varieties, the canonical height $\h_A(V)$ rarely decomposes as a sum of non-negative local height terms. This difficulty is already visible in earlier work of Zhang \cite{Zha95b} where he proves the Bogomolov conjecture under some restricted hypotheses. It is therefore natural to ask whether the ultrafilter degeneration argument can be applied directly to the global height $\h(V)$ itself, rather than its local components. It turns out such an argument is possible, using the theory of globally valued fields by Ben Yaacov and Hrushovski \cite{BH22}, combined with the theory of adelic curves by Chen and Moriwaki \cite{CM20}.

\subsection{Globally Valued Fields and Adelic Curves} \label{sec: IntroGVF}

The notion of a globally valued field arose from studying the Weil height function $h$ on $\ovl{\mathbb{Q}}$ from a model theory perspective. Following \cite{GVFfoundations}, we define a globally valued field (GVF) to be a field $K$ with a height function $h: K^n \to \mathbb{R} \cup \{ -\infty\}$ satisfying:

\begin{enumerate}
\item $h(x) = -\infty$ if and only if $x = 0$ and $h(1,1) = 0$;

\item $h(x) = 0$ for all $x \in K^{\times}$;

\item $h(\sigma(x)) = h(x)$ for all $x \in K^n$ and permutation $\sigma \in S_n$;

\item $h(x \otimes y) = h(x) + h(y)$ for all $x \in K^n$ and $y \in K^m$, where $x \otimes y \in K^{mn}$ denotes the Segre product;

\item $h(x) \leq h((x,y))$ for all $x \in K^n$ and $y \in K^m$ where we view $(x,y)$ as an element of $K^{n+m}$;

\item there exists $e > 0$, called the archimedean error, such that $h(x+y) \leq h((x,y)) + e$ for all $x \in K^n$ and $y \in K^m$;
\end{enumerate}

Using property (2) and (4), it follows that $h$ descends to a function on $h: \mathbb{P}^n(K) \to \mathbb{R} \cup \{-\infty\}$ and hence can be viewed as a Weil height function. By abuse of notation, for $x \in K$ we will let $h(x) = h([x:1])$. Furthermore, globally valued fields are closed under ultraproducts. Let $\omega \in \beta \mathbb{N}$ be an ultrafilter on $\mathbb{N}$. Given a sequence $(K_n)_{n \geq 1}$ of GVFs with height functions $(h_n)_{n \geq 1}$, we may form its ultraproduct 
$$K_{\omega}:= \left\{ (x_n) \in \prod_{n=1}^{\infty} K_n \mid \sup_{n} h_n(x_n) < +\infty \right\}/\sim ,$$
where we say that $(x_n) \sim (y_n)$ if the set $\{n \mid x_n = y_n\}$ is $\omega$-large. Then $K_{\omega}$ is a field, and we may define a height function 
$h_{\omega}: K_{\omega}^m \to \mathbb{R} \cup \{-\infty\}$ by $$h_{\omega}((x_{n,1}),\ldots,(x_{n,m})) = \lim_{\omega} h_n(x_{n,1},\ldots,x_{n,m}).$$
It is then easy to verify that $(K_{\omega},h_{\omega})$ is a GVF.
\par 
Hrushovski observed that the Bogomolov conjecture admits a natural formulation over arbitrary globally valued fields. Furthermore, using ultraproducts, one can show that the new gap principle of Gao--Ge--Kühne \cite{GGK21} is equivalent to the Bogomolov conjecture over all globally valued fields of characteristic zero. Thus, uniform Bogomolov-type results may be approached by proving the Bogomolov conjecture in the broader setting of globally valued fields.
\par 
A priori, the data of a global height function $h$ for a GVF may seem too unconstrained to be useful. However when $K$ is countable, such height functions always admit a decomposition into local terms as proven in \cite{GVFfoundations}. Namely, there exists a measure space $(\Omega, \mathcal{A}, \nu)$, where $\Omega$ is a set, $\mathcal{A}$ is a $\sigma$-algebra on $\Omega$ and $\nu$ is a measure for $\mathcal{A}$, with a map $\phi: \omega \mapsto | \cdot |_{\omega}$ from $\Omega$ to the set of absolute values $M_K$ of $K$, for which the function $\omega \mapsto \log |x|_{\omega}$ is integrable for each $x \in K^{\times}$. The global height $h$ on $K^n$ then admits the following local decomposition 
$$h(x_1,\ldots,x_n) = \int_{\Omega} \max\{ \log |x_1|_{\omega},\ldots, \log |x_n|_{\omega}\} d \nu(\omega).$$
Height functions of this kind were also studied independently by Chen and Moriwaki \cite{CM20}, who refer to the data $(K,\Omega,\mathcal{A},\nu)$ as an adelic curve. In a series of works \cite{CM20, CM21, chen2024positivity}, Chen and Moriwaki developed Arakelov theory over adelic curves and showed that many fundamental results from classical Arakelov geometry extend to this setting, including the arithmetic Hilbert--Samuel theorem and the equidistribution of small points. This raises the possibility of transporting proofs of the Bogomolov conjecture over number fields and function fields to general adelic curves, and hence to GVFs. Indeed, if the GVF is an archimedean GVF ($h(2) \not = 0$), then Chen--Moriwaki \cite{chen2022equidistribution} already proved the Bogomolov conjecture in this case using the strategy of Szpiro--Ullmo--Zhang \cite{SUZ97}. 
\par 
The first successful implementation of this strategy for general GVFs is due to Hultberg \cite{Hul26}. Following the approach of Luo and Yu in the function-field setting \cite{LY25}, he deduced the Bogomolov conjecture for semiabelian varieties over characteristic-zero GVFs from the corresponding result for abelian varieties. As a consequence, the new gap principle for semiabelian varieties over $\ovl{\mathbb{Q}}$ follows from the work of Gao--Ge--Kühne \cite{GGK21} in the abelian variety setting.
\par 
There are two obstacles to adapting the Gubler--Yamaki strategy to globally valued fields. The first is that the results of Gubler and Yamaki are formulated over discretely valued non-Archimedean fields, whereas the globally valued setting requires us to work over arbitrary valued fields. This difficulty is largely technical: Gubler and Stadlöder \cite{gubler2023monge} have extended Gubler’s results to general non-Archimedean valued fields. In Section \ref{sec: non-archimedean canonical measures}, we explain why Yamaki’s arguments also admit a corresponding extension.
\par 
The second obstacle is more substantial. The arguments of Gubler and Yamaki require the existence of a place of bad reduction, which need not be available over a general globally valued field. Over function fields, this was the principal obstruction to proving the geometric Bogomolov conjecture in full generality, as emphasized by Yamaki \cite{Yam18}. For divisors, Yamaki \cite{Yam17} treated the everywhere good reduction case by a separate argument, while for arbitrary subvarieties this case was later resolved by Xie and Yuan \cite{XY22}, completing the proof of the geometric Bogomolov conjecture.
\par 
The approaches of Yamaki and Xie--Yuan, however, a priori rely heavily on features specific to heights over function fields, and it is not clear how to extend them to an arbitrary adelic curve. We instead give new arguments for both cases that are more Arakelov-theoretic, and hence hold over GVFs. 
\par 
Our new proof consists of two parts. We first prove the Bogomolov conjecture for divisors over GVFs, before showing the case of subvarieties reduces directly to the case of both curves and divisors, and hence just divisors. In Xie--Yuan \cite{XY22}, an inductive argument reduces the case of subvarieties to the case of points via the Manin--Mumford conjecture \cite{Ray83, Hru01}. As height zero points in a general GVF need not be torsion or isotrivial, we cannot invoke the Manin--Mumford conjecture which causes Xie--Yuan's proof breaks down. In particular, our proof gives a new proof of the Manin--Mumford conjecture in positive characteristic, similar in nature to Moriwaki's proof of Manin--Mumford over $\bb{C}$ using the Moriwaki height \cite{Mor00}.

\subsection{Bogomolov Conjecture for Divisors}
We now explain our proof in more detail for the case of divisors. Over a function field $K$, Yamaki \cite{yamaki2013geometric}, extending the work of Gubler in the totally degenerate case \cite{gubler2007bogomolov}, proved that if $X$ is a subvariety of $A$, with finite stabilizer, satisfying $\dim X \geq \dim b(A_v)$, then $X$ does not have a dense set of small points. Here, $b(A_v)$ denotes the good reduction part of $A$ at $v$, which is given by Raynaud's uniformization \cite[Definition 4.3]{yamaki2013geometric}. 
\par 
If $X$ is a divisor with finite stabilizer, and $A$ has bad reduction at $v$, then $\dim b(A_v) \leq \dim A - 1$ and we certainly get $\dim X \geq \dim b(A_v)$. This leaves us with the case where $A$ has everywhere good reduction. Our new idea in this case is motivated by a formula by de Jong--Shokrieh \cite[Theorem A]{deJong2022faltings}
\begin{equation} \label{eq: IntroEqn1}
h_{\Fal}(A) = 2g \h_{A,L}(\Theta) + 2 \sum_{v \in M_K} I(A_v,\lambda_v).
\end{equation}
Here, $(A,L)$ is a principally polarized abelian variety with $L$ symmetric and $\Theta = \divv(s)$ is the divisor of the unique (up to scaling) global section of $L$. The quantity $I(A_v,\lambda_v)$ is some tropical moment associated to the Berkovich analytification $A^{\an}_v$ of $A$ over $v$ and vanishes if and only if $A$ has potential good reduction at $v$. In the number field setting, this generalizes earlier work of Autisser \cite{Aut06} in the everywhere good reduction case.
\par 
If $A$ has everywhere good reduction, then (\ref{eq: IntroEqn1}) reduces immediately to $h_{\Fal}(A) = 2g \h_{A,L}(\Theta)$. This gives a uniform height gap for theta divisors as desired. Using work of Zhang on the admissible pairing of curves \cite{Zha93}, it is already possible to deduce the uniform Bogomolov conjecture for curves in Jacobians \cite{Kuh21, Yua24}. 
\par 
For a general effective divisor $D$, we can try to generalize the mechanism behind (\ref{eq: IntroEqn1}). It turns out the essential part we need comes from Moret-Bailly's key formula \cite{MB85}, and in particular from Bost's reformulation of it \cite{Bos96}. If $D$ is a symmetric divisor with canonical height zero, then if $\ovl{L}$ is the line bundle $O(D)$ with its canonical metric, and $s$ is the canonical section of $L$ representing $D$, we must have 
\begin{equation} \label{eq: IntroEqn3}
\sum_{v \in M_K} \log \sup_{x \in A_v^{\an}} |s(x)|_v  = 0.
\end{equation}
On the other hand, Moret-Bailly's key formula implies that for the adelic vector bundle $\pi_* \ovl{L}$, where each global section $s$ has the norm $||s||_v = \sup_{x \in A_v^{\an}} |s(x)|_v$, we have the identity
\begin{equation} \label{eq: IntroEqn2}
\wh{\mu}(\pi_* \ovl{L}) = -\frac{d}{2} h_{\Fal}(A),
\end{equation}
where $\wh{\mu}$ is the average slope and $d = \dim \pi_*\ol{L}.$ (\ref{eq: IntroEqn2}) can be interpreted as saying that the average value, for a basis of sections, of $\sum_{v \in M_K} \log \sup |s(x)|_v$ for a section $s$ 
should be proportional to the Faltings height, which is in tension with (\ref{eq: IntroEqn3}). If $K$ is of characteristic zero, semistability of $\pi_* \ovl{L}$ \cite[Proposition A.3]{Bos96b} immediately gives a contradiction. 
\par 
In Section \ref{sec: DivisorGVF}, we prove a generalization of the key formula for abelian varieties over adelic curves with everywhere good reduction, using work of Tang \cite{Tan14}. For simplicity, we will assume that $A$ is traceless over the height zero subfield. We then show that if $D$ is a symmetric divisor with $\h_A(D) = 0$, then we can construct many small-average sections of $\ovl{L}$, where $L = O(D)$. The small-average sections gives us a map $A \to \bb{P}^N$ that factors through $A \to B \xhookrightarrow{} \bb{P}^N$ where $\pi: A \to B$ is an isogeny. This then implies that $h_{\Fal}(B) = 0$, which contradicts our assumption that $A$ is traceless. 

\subsection{Bogomolov Conjecture for Subvarieties}
For a general subvariety $V \subseteq A$, we reduce the Bogomolov conjecture for $V$ to the case of divisors. For simplicity let's assume that $A$ is simple. If $V = C$ is a curve, then $D = \underbrace{C + \cdots + C}_{g-1 \text{ times}}$ is a divisor, and if $C$ has canonical height zero then so does $D$. 
\par 
Now let's say $V$ is a surface. We first consider the case where $\dim A$ is odd. Then if $r = \lceil \frac{\dim A}{2} \rceil$, the morphism $\pi: V^r \to A$ given by addition is surjective and its generic fiber has dimension $1$. In the case of function fields, if $A$ has everywhere good reduction, Xie--Yuan \cite{XY22} shows that for a general torsion point $x \in A(\ovl{K})$, the fiber $\pi^{-1}(x)$ also has canonical height zero. This contradicts the case of curves for us. In general for GVFs, we endow the function field $K(A)$ with a GVF structure, coming from the polarization induced by our canonical line bundle. Then the generic fiber of $V^r \to A$ is a curve over $K(A)$ with canonical height zero inside $A_{K(A)}^r$, giving us a contradiction.  
\par 
Now if $\dim A$ is even, say $\dim A = 4$, then $\pi: V^2 \to A$ is generically finite and the generic fiber is zero dimensional. As height zero points need not be torsion for an arbitrary GVF, we cannot use the above strategy. We instead consider $\delta: V^2 \to A$, given by $(x,y) \mapsto x-y$. We have two adelic line bundles on $V^2$, given by $\ovl{P} = p_1^*\ovl{L} + p_2^*\ovl{L}$ and $\ovl{N} = \delta^* \ovl{L}$. The line bundle $\ovl{P}$ seems to be more positive than the bundle $\ovl{N}$ but the height of $V^2$ with respect to both is zero. We can view the strategy of Zhang \cite{Zha98} as exploiting this fact: by equidistribution, we must have $c_1(\ovl{P})^{\dim V^2}$ and $c_1(\ovl{N})^{\dim V^2}$ be proportional to each other, but $c_1(\ovl{N})^{\dim V^2}$ must vanish along the diagonal as $\delta$ contracts it, giving us a contradiction.  
\par 
In the everywhere good reduction case, the measures $c_1(\ovl{P})^{\dim V^2}$ and $c_1(\ovl{N})^{\dim V^2}$ no longer give any useful information. However, there is still a height equidistribution property: for any effective divisor $D \subseteq V^2$, we have 
$$h_{\ovl{P}}(D) = h_{\ovl{N}}(D) = 0.$$
The reader may see Section \ref{sec: HeightEquidistribution} for more details. In the classical case, this property can be deduced from a result of Chambert-Loir and Thuillier \cite{CLT09}, that computes the limit of $h_{\ovl{M}}(x_n)$ for a generic sequence of small points $(x_n)$. This motivates us to call it a height equidistribution property. Now if $D$ was an exceptional divisor that is contracted by $\delta: V^2 \to A$, we must have $h_{\ovl{N}}(D) = 0$ and hence $h_{\ovl{P}}(D) = 0$. A generic linear projection $A^2 \to A$ converts $D$ to a height zero divisor inside $A$, contradicting the divisor case. 
\par 
The main difficulty now is to construct this exceptional divisor. The idea is to use a lemma from birational geometry \cite[Corollary 2.63]{KM08}, which states that for a proper birational morphism $\pi: X' \to X$, if $\pi$ is not an isomorphism then it must have an exceptional divisor. Let $Y = V^2$. Then the map $\delta: Y \to A$ allows us to view $K(Y)$ as a finite extension of $K(A)$. Let's assume for simplicity that $K(Y)/K(A)$ is separable and Galois with Galois group $G$. If $G$ already acts on $Y$, then we may take $Y' = Y/G$, which will have a birational morphism $\delta': Y' \to A$. However this morphism cannot be an isomorphism as $Y \to Y'$ is finite but $\delta: Y \to A$ is not finite. Hence by \cite[Corollary 2.63]{KM08}, there must be an exceptional divisor as desired. In general, we modify $Y$ to obtain $X$ so that $G$ acts on $X$, but $X$ still embeds into $Y_A^s$ for some positive integer $s$, see Proposition \ref{ExceptionalDivisor1}. The idea for this Proposition was supplied by GPT. As $X$ embeds into $A^s$, a height zero divisor $D$ for $X$ will become a height zero subvariety of $A^s$ of smaller dimension, which allows us to obtain a contradiction by induction. 

\subsection{Outline of the paper} We begin in Section 2 by reviewing the theory of adelic curves, following the work of Chen--Moriwaki \cite{CM20, CM21, chen2024positivity}. The theory of adelic line bundles on quasi-projective varieties developed by Yuan--Zhang \cite{yuan2021adelic} is recalled in Section 3. Section 4 then introduces globally valued fields (GVFs), following \cite{GVFfoundations}, and explains their relationship with both Yuan--Zhang and Chen--Moriwaki adelic line bundles. Section 5 revisits the strategy of Gubler and Yamaki in the bad reduction case and show, following Gubler--Stadloader \cite{GS23}, that their arguments extend to arbitrary valued fields.
\par 
With this background in place, Section 6 establishes the Bogomolov conjecture for divisors over GVFs. We then extend in Section 7 some results from Arakelov theory that are necessary for us, previously known over number fields and function fields, to the general setting of GVFs. The Bogomolov conjecture for arbitrary subvarieties over GVFs is proved in Section 8. Finally, Section 9 applies these results to deduce uniform height gaps for subvarieties in positive characteristic.

\subsection{AI Disclosure} The third author used GPT throughout the preparation of the paper to check and verify claims, and also to proofread the paper. All ideas in the paper were due to the authors, with the exception of Proposition \ref{ExceptionalDivisor1} whose main idea was suggested to us by GPT.

\subsection{Acknowledgements} We would like to thank Nuno Hultberg and Michał Szachniewicz for valuable discussions regarding the theory of GVFs. We would also like to thank Ziyang Gao and Niki Myrto Mavraki for helpful discussions regarding the Bogomolov conjecture.

\section{Adelic Curves}
Classically, Arakelov theory was usually studied over a base field $K$ which was either a number field or a global function field. In \cite{Mor00}, Moriwaki introduced what is now known as the Moriwaki height when $K$ is a finitely generated field over $\bb{Q}$. The advantage of the Moriwaki height is that it still satisfies the Northcott property \cite[Section 4]{Mor00}, unlike the usual function field height for $K$. Moriwaki proceeded to show that many results from Arakelov theory still holds for this Moriwaki height, and in particular he was able to adapt Ullmo \cite{Ulm98} and Zhang's \cite{Zha98} strategy for the Bogomolov conjecture. This gives a new proof of the Manin--Mumford conjecture over $\bb{C}$, as Zhang's proof of the Bogomolov conjecture only applies over $\ovl{\bb{Q}}$. 
\par 
In \cite{CM20}, Chen and Moriwaki defined a very general object known as an adelic curve. Adelic curves contains almost all classical examples as special cases, such as Moriwaki heights and also heights coming from product formula fields. Chen and Moriwaki \cite{CM20, CM21, chen2024positivity} then show a very robust version of Arakelov theory exists in this very general setting. In this section, we will review some of the basic definitions and theory of adelic curves and adelic line bundles on these curves. Recently, S{\'e}dillot \cite{Sed26} proposed a more general setting of topological adelic curves, which aims to place Nevanlinna theory in this framework too. 

\subsection{Basics on Adelic Curves}

We start by introducing adelic curves. We follow \cite[Chapter 3]{CM20}, where this notion was introduced.

\begin{defn}
An \textbf{adelic curve} is the data $(K, (\Omega, \mathcal{A}, \nu), \phi)$ of a field $K,$ a measure space $(\Omega, \mathcal{A}, \nu),$ and a map $\phi: \Omega \rightarrow M_K,$ where $M_K$ is the set of all absolute value functions on $K,$ satisfying the following properties:
\begin{enumerate}
    \item $\mathcal{A}$ is an $\sigma$-algebra and $\nu$ is a measure on $(\Omega, \mathcal{A}),$
    \item For any $a \in K^{\times},$ the function $\omega \mapsto \log \phi(\omega)(a)$ is integrable with respect to $\nu$ (and thus is also $\mathcal{A}$-measurable).
\end{enumerate}
We say that this adelic curve is \textbf{proper} if for all $a \in K^{\times},$ $\int\limits_{\Omega} \log \phi(\omega)(a) d\nu = 0.$
\end{defn}
We will simplify notation and write $|\cdot|_{\omega}$ for the absolute value $\phi(\omega).$ Similarly, we write $K_{\omega}$ for the completion of $K$ with respect to $|\cdot|_{\omega}.$  
\par For a proper adelic curve, we may define the height $h(x)$ of an element $x \in K^{\times}$ by 
$$h(x) := \int_{\Omega} \max\{\log|x|_{\omega}, 0\} d \nu(\omega).$$
\par As an example, let $K$ be a number field. Then, we can let $\Omega$ be the set of equivalence classes of nontrivial absolute values of $K,$ and let $\mathcal{A}$ be the power set of $\Omega.$ We then have the function $\phi: \Omega \rightarrow M_K$ which sends an equivalence class to the absolute value which is normalized in the following way:
\begin{enumerate}
    \item If $|\cdot|_{\omega}$ is nonarchimedean and corresponds to a place over $p,$ then we normalize such that $|p|_{\omega} = 1/p.$
    \item If $|\cdot|_{\omega}$ is archimedean, we normalize such that $|2|_{\omega} = 2.$
\end{enumerate}
We then consider the measure $\nu$ such that the measure of $|\cdot|_{\omega}$ is $[K_{\omega}: \mathbb{Q}_p]$ if it nonarchimedean and corresponds to a place over $p$ and is $[K_{\omega}:\mathbb{R}]$ if is archimedean. Then, this data forms a proper adelic curve, owing to the product formula. 
\par In general, given an adelic curve structure on a field $K,$ and an algebraic extension $L/K,$ one can form a compatible adelic curve structure on $L,$ by letting $\Omega_L$ consist of the absolute values of $L$ that extend the absolute values $\phi(\omega)$ for $\omega \in \Omega_K;$ see \cite[\S 3.3]{CM20} for more details. 
\par One useful result is that the condition of an absolute value being trivial or archimedean is measureable, as follows.
\begin{proposition}[\cite{CM20}, Propositions 3.1.1, 3.1.3]
Let $(K, (\Omega, \mathcal{A}, \nu), \phi)$ be an adelic curve. Let $\Omega_{\infty}$ be the set of $\omega \in \Omega$ for which $\phi(\omega)$ is archimedean, and let $\Omega_0$ be the set of $\omega \in \Omega$ for which $\phi(\omega)$ is trivial. Then, $\Omega_{\infty}$ is $\nu$-measurable, and if $K$ is countable, then $\Omega_0$ is as well. 
\end{proposition}
\par We say that our adelic curve is \textbf{archimedean} if $\nu(\Omega_{\infty}) > 0$ and \textbf{non-archimedean} otherwise. We may detect this by looking at whether $h(2)$ is positive or not. 
\par As we can see in the above proposition, the theory is nicer when $K$ is assumed to be countable. Throughout, we will mainly be working in the countable case.

\subsection{Chen--Moriwaki Adelic Line Bundles}
For this subsection, fix a proper adelic curve $S = (K, (\Omega, \mathcal{A}, \nu), \phi).$ In \cite{CM20}, Chen--Moriwaki define the notion of adelic line bundles as follows.

\begin{defn}[\cite{CM20}, Definition 4.1.28]
An \textbf{adelic line bundle} is a one-dimensional vector space $V$ along with metrics $\|\cdot\|_{\omega}$ on $V \otimes_K K_{\omega}$ for each $\omega \in \Omega,$ such that the family of these metrics are dominated (\cite[Definition 4.1.2]{CM20}) and measurable (\cite[\S 4.1.3]{CM20}).    
\end{defn}

More generally, they define the notion of an adelic line bundle on a projective scheme over $K.$

\begin{defn}[\cite{CM20}, Definition 6.2.1]
Let $X$ be a projective scheme over $\Spec K$. An \textbf{adelic line bundle} on $X$ is an invertible $\mcO_X$-module $L$ equipped with a metric family $\varphi$ which is dominated (\cite[Definition 6.1.9]{CM20}) and measurable (\cite[Definition 6.1.27]{CM20}).
\end{defn}

\begin{remark}
Note that already over a number field $K$, the Chen--Moriwaki line bundles are more general than the ones introduced by Zhang \cite{Zha95b} as we only assume a dominated condition on the metrics, instead of triviality at all but finitely many places. As such, we will distinguish these two notions by calling the Chen--Moriwaki version CM adelic line bundles. Later in the paper we describe how these two notions of adelic line bundles are related.
\end{remark}

Adelic line bundles on adelic curves also admit the notion of Arakelov degree.

\begin{defn}[\cite{CM20}, Definition 4.3.1]
Let $V$ be an adelic line bundle on $S.$ Its Arakelov degree is defined to be
\begin{align}
\label{eqn:arakelovdegree}
    \int_{\Omega} -\log \|s\|_{\omega} d\nu. 
\end{align}
for any nonzero $s \in V.$
\end{defn}
Because our adelic curve is proper, this expression is independent of the choice of nonzero section, and so this is actually well-defined.
\par We may similarly define the notion of an adelic Cartier divisor. For the rest of the subsection, fix a geometrically integral projective $K$-scheme $X.$ 
\begin{defn}[\cite{CM20}, Definition 2.5.2]
Let $\omega \in \Omega$ and $D$ a Cartier divisor on $X.$ Let $X^{\an}_{\omega}$ be the Berkovich analytification of $X$ with respect to the valued field $(K, \phi(\omega)).$ A \textbf{Green function} for $D$ is a continuous function $g$ on some Zariski-dense open $U$ of $X$ such that for any nonempty Zariski open of $X$ where $D$ is cut out by the function $f,$ $g + \log f$ extends to a continuous function on $U^{\an}.$ 
\end{defn}
\begin{defn}[\cite{CM20}, \S 6.2.3]
An \textbf{adelic Cartier divisor} is the data of a Cartier divisor $D$ on $X$ and a family of compatible Green functions $\{g_{\omega}\}_{\omega \in \Omega}$ such that the associated metric family given by $e^{-g_{\omega}(\bullet)}$ is dominated and measureable (in the same sense as for adelic line bundles).
\end{defn}
Now, suppose that $L$ is an adelic line bundle on $X.$ Choosing a section $s$ of $L$ naturally gives rise to an adelic Cartier divisor $\ovl{D}$ by taking $g_{\omega}(x) = \log |s(x)|_{\omega}$.
\par We say that our adelic Cartier divisor is \textbf{semipositive} if the metric $| \cdot |_{\omega}$ is semipositive (see \cite[\S 2.3.1]{CM20}) for all $\omega \in \Omega$ and is \textbf{integrable} if $\ovl{D} = \ovl{D}_1 - \ovl{D}_2$, where each $\ovl{D}_i$ is an ample Cartier divisor with semipositive Green functions. 
\par Let $d = \dim X$. Given $d+1$ integrable adelic Cartier divisors $\ovl{D}_0,\ldots,\ovl{D}_d$ such that the Cartier divisors intersect properly, we may define a local intersection number $(\ovl{D}_0 \cdot \ovl{D}_1 \cdots \ovl{D}_d)_{\omega}$ for each $\omega \in \Omega$ \cite[Definition 3.5.1]{CM21}. Over a number field, this recovers the classical local intersection number \cite[Section 4.5]{CM21}. The function $\omega \mapsto (\ovl{D}_0 \cdots \ovl{D}_d)_{\omega}$ turns out to be integrable \cite[Theorem 4.2.10]{CM21}, which allows us to define the global intersection number as
$$\ovl{D}_0 \cdots \ovl{D}_d =  \int_{\Omega} (\ovl{D}_0 \cdots \ovl{D}_d)_{\omega} d \nu(\omega).$$
This gives a well-defined global intersection number for integrable adelic line bundles $\ovl{L}_0,\ldots,\ovl{L}_d$. 
\par 
For reference, we shall state some important results in Arakelov theory that have analogues over countable adelic curves. There is an Nakai--Moishezon criterion \cite[Section 7]{CM20}, an arithmetic Hilbert--Samuel formula \cite[Section 5.5]{chen2024positivity}, and an equidistribution theorem for small points \cite[Section 8.11]{chen2024positivity}. Using these results, Chen and Moriwaki are able to prove the Bogomolov conjecture for a countable archimedean adelic curve \cite[Section 9.4]{chen2024positivity}.

\section{Adelic Line Bundles on Quasi-Projective Varieties}\label{sec: adelic line bundles on quasi-projective varieties}
We now review the theory of adelic line bundles on quasi-projective varieties following Yuan--Zhang \cite{yuan2021adelic}. In the early development of Arakelov theory, only model line bundles were considered, i.e. hermitian line bundles $\ovl{\mathcal{L}}$ on an arithmetic variety $X$ over $\Spec O_K$, see for example \cite{Fal84, GS90}. In \cite{Zha95b}, Zhang, motivated by the definition of the Call--Silverman canonical height \cite{CS93}, introduced adelic line bundles as a limit of model line bundles. 
\par 
The theory of Yuan--Zhang can be seen as taking one step further, defining adelic line bundles on a quasi-projective variety $U/K$ as a limit of adelic line bundles on a sequence of projective compactifications $(X_n)_{n \geq 1}$ of $U$. Such adelic line bundles arise naturally when considering families over a base $U$ and play an important role in uniformity results for both arithmetic dynamics \cite{DM24} and diophantine geometry \cite{Yua24}. In this section, we review basic definitions and properties from \cite{yuan2021adelic}.

\subsection{Yuan--Zhang Adelic Line Bundles}
We first define the notion of adelic line bundles in the sense of \cite{yuan2021adelic}, which we will call Yuan--Zhang adelic line bundles. 
The following discussion follows \cite[\S 1.1]{yuan2021adelic}.

Let $k$ be $\Z$ or a field, and let $X/k$ be quasiprojective; that is, $X$ is an integral and open subscheme of a projective and flat $k$-scheme. Let $X^{\an}$ be the Berkovich analytification of $X$ over the Banach ring $(\bb{Z}, | \cdot |_{\infty});$ see \cite[Section 3]{yuan2021adelic} for more information.

\begin{defn}
We denote by $\widehat{\Pic}(X^{\an})$ to be the group (under tensor product) of isomorphism classes of pairs $(L, \abs{\cdot})$ of a line bundle $L$ on $X$ and a continuous metric $\abs{\cdot}$ of $L$ on $X^{\an},$ where isomorphisms are given by isomorphisms of line bundles which preserve the metric. We call such pairs \textbf{metrized line bundles.}
\end{defn}
\begin{remark}
\cite{yuan2021adelic} also defines the category of metrized line bundles. However, we will only need to use the group, so we will not concern ourselves with this definition. 
\end{remark}

We now turn to the construction of adelic line bundles, which are a particular subcategory of $\widehat{\Pic}(X^{\an}).$

\begin{defn}
We say that a metrized line bundle $L$ on $X$ is a \textbf{model adelic line bundle} if there exists a pair $(\mathcal{X}, \overline{\mcL}),$ where $\mathcal{X}$ is a projective model of $X$ over $k$ and $\overline{\mcL}$ is
\begin{enumerate}
    \item a line bundle on $\mathcal{X}$ extending $L^{\otimes e}$ for some positive integer $e,$ if $k$ is a field, or
    \item a hermitian line bundle $(\mcL, \abs{\cdot})$ on $\mathcal{X}$ if $k = \mathbb{Z},$ where $\mcL$ is a line bundle on $\mathcal{X}$ extending $L^{\otimes e}$ for some positive integer $e$ and $\abs{\cdot}$ are continuous metrics of $\mcL(\mathbb{C})$ in $\mathcal{X}(\mathbb{C}),$ such that these metrics are invariant under complex conjugation.  
\end{enumerate}
The subcategory of $\widehat{\Pic}(X^{\an})$ consisting of model adelic line bundles is denoted $\widehat{\Pic}(X/k)_{\mod}.$
\end{defn}

We then define a \textbf{boundary topology} on the space $C(X^{\an})$ of real-valued continuous functions on $X^{\an}.$ This is done by fixing a Green function $g$ which is continuous on $X^{\an}$ and has logarithmic poles along $\ovl{X}^{\an} \setminus X^{\an}$ for some projective model $\ovl{X}$ of $X$. Then, we say that a sequence $\{f_n\}_{n \geq 1}$ of functions on $X^{\an}$ converges to $f$ if $|f - f_n| \leq c_n |g|$ for some sequence $c_n$ tending to $0$. This boundary topology is independent of the choice of $g;$ see \cite[Lemma 2.4.1]{yuan2021adelic}, for instance.
\par Using this topology, we can define Cauchy sequences of model adelic line bundles (see \cite[Section 1.1.3]{yuan2021adelic}). The category of model adelic line bundles is not closed under taking limits of Cauchy sequences with respect to the boundary topology, but the limits exist in the larger category $\widehat{\Pic}(X^{\an})$.
\begin{defn}
    An \textbf{adelic line bundle} is an object of $\widehat{\Pic}(X^{\an})$ which is isomorphic to the limit of a sequence in $\widehat{\Pic}(X/k)_{\mod}$. We denote the category of adelic line bundles as $\widehat{\Pic}(X/k)$. 
\end{defn}

One can extend the notion of adelic line bundles from varieties over $k$ to varieties over finitely generated fields $F$ over $\Frac(k).$ Suppose that $X/F$ is a quasiprojective variety. We then define \[\widehat{\Pic}(X/k) := \lim\limits_{U \rightarrow V} \widehat{\Pic}(U/k),\] where we take the limit over flat morphisms $U \rightarrow V$ of quasiprojective varieties over $k,$ where $V$ has generic point $\Spec F$ and the generic fiber of this morphism is $X \rightarrow \Spec F.$ Note that if $X$ is a quasi-projective variety over $\bb{Q}$, then $\wh{\Pic}(X/\bb{Q})$ and $\wh{\Pic}(X/\bb{Z})$ are two different categories. An object in $\wh{\Pic}(X/\bb{Q})$ is a limit of geometric line bundles of compactifications of $X$ over $\bb{Q}$, whereas an object of $\wh{\Pic}(X/\bb{Z})$ is a limit of adelic line bundles coming from $\wh{\Pic}(\mathcal{U}/\bb{Z})$ where $\mathcal{U}$ are quasi-projective varieties over $\bb{Z}$.

\begin{remark}
Yuan--Zhang later provide a different definition of adelic line bundles, namely as equivalence classes of adelic divisors (see \cite[\S 2.4]{yuan2021adelic}). However, this construction is equivalent to the definition we use, which is taken from \cite[\S 1]{yuan2021adelic} (see also \cite[\S 3.4]{yuan2021adelic}).
\end{remark}

We now define the notion of a strongly nef (snef), nef and integrable adelic line bundle. For a model adelic line bundle $(\mathcal{X}, \ovl{\mathcal{L}})$, if $k$ is a field, we say that it is \textbf{nef} if $\mathcal{L}$ is a nef line bundle on the model $\mathcal{X}$. If $k = \bb{Z}$, then we ask that the hermitian metric $| \cdot |$ is semipositive (that is, has positive curvature current), and that $\wh{\deg}(\ovl{\mathcal{L}} \vert_C) \geq 0$ for any one-dimensional closed integral subscheme $C \subseteq \mathcal{X};$ see \cite[\S A.4.1]{yuan2021adelic}. This latter condition implies that $\mathcal{L}$ is nef if we consider vertical $C$'s, but also contains additional information from horizontal curves. 
\par Now given an adelic line bundle $\ovl{L} \in \wh{\Pic}(X^{\an})$, we say that it is \textbf{strongly nef} if it is the limit of nef model adelic line bundles. We say it is \textbf{nef} if there exists a strongly nef adelic line bundle $\ovl{M}$ such that $ a \ovl{L} + \ovl{M}$ is strongly nef for all $a \geq 0$. We say that it is \textbf{integrable} if it can be written as the difference of two strongly nef adelic line bundles. We will denote these subcategories by $\wh{\Pic}_{\operatorname{snef}}(X/k), \wh{\Pic}_{\operatorname{nef}}(X/k)$ and $\wh{\Pic}_{\mathrm{int}}(X/k)$ respectively. 

\subsection{The Deligne Pairing and Heights}
We now recall the Deligne pairing defined in \cite{yuan2021adelic}, following \cite[\S 1.2]{yuan2021adelic}. The Deligne pairing was first introduced by Deligne in \cite{Del87} for relative curves, and then subsequently extended to higher relative dimensions by Elkik \cite{Elk89}. An alternative approach was given by Ducrot \cite{Duc05}.
\par Suppose we are given a projective flat morphism of quasiprojective varieties $X \rightarrow Y$ over $k$ of relative dimension $n.$ If we are given a line bundle $L$ on $X$ which comes from a projective model $\tilde{X}$ of $X,$ one can define the Deligne pairing to produce a line bundle on $Y$ as follows. Using Nagata compactification and the Raynaud--Gruson flattening theorem (\cite[Theorem 5.2.2]{GR71}), we can pick compatible compactifications $\mc X$ of $X$ and $\mc Y$ of $Y$ such that there is a projective flat morphism $\mc X \to \mc Y$ and $L$ comes from $\mathcal{X}.$ Then, we can use the classical Deligne pairing, which is a map $\Pic(\mc X)^{n+1} \to \Pic(\mc Y)$, where $n$ is the relative dimension of $\mc X \to \mc Y.$ 
\par In other words, the Deligne pairing on projective models gives us a Deligne pairing on model adelic line bundles. In \cite{yuan2021adelic}, they prove that this Deligne pairing on model adelic line bundles extends to integrable adelic line bundles, as follows.

\begin{thm}[Theorem 1.2.1, \cite{yuan2021adelic}]
Let $X \to Y$ be a projective flat morphism of quasiprojective varieties over $k$ of relative dimension $n$, where $k$ is either $\Z$ or a field. Suppose $Y$ is normal. 
The Deligne pairing on model adelic line bundles induces a symmetric and multilinear functor \[\widehat{\Pic}(X)^{n+1}_{\mathrm{int}} \rightarrow \widehat{\Pic}(Y)_{\mathrm{int}}\] 
compatible with base changes $Y' \to Y$ where for $Y'$ any quasiprojective normal variety over $k$ such that $X \times_Y Y' $ is integral.
\end{thm}

Now, suppose that $\overline{L}$ is a line bundle in $\widehat{\Pic}(X/k)_{\mathrm{int}},$ and $F$ is the function field of $Y.$ Then, \cite{yuan2021adelic} define a vector-valued height function \[X(\overline{F}) \rightarrow \widehat{\Pic}(F/k)_{\mathrm{int}, \mathbb{Q}}\] as follows. Consider the image of $x: \Spec \overline{F} \rightarrow X_F;$ this will be some subscheme $x'$ of $X_F,$ which spreads out to a projective and flat scheme over an open of $Y.$ We may then apply the Deligne pairing to this open subscheme to obtain $\langle \overline{L}|_{x'} \rangle \in \widehat{\Pic}(F/k)_{\mathrm{int}, \mathbb{Q}},$ and thus define \[h_{\overline{L}}(x) = \frac{\langle \overline{L}|_{x'} \rangle}{\deg x'},\] where $\deg x'$ is the degree of the residue field of $x'$ over $F.$

We can similarly define such a height function for subvarieties $Z$ of $X_F$ by defining 
\[
h_{\overline{L}}(Z) = \frac{\langle \ol{L}|_{Z'} \rangle^{\dim Z + 1}}{(\dim Z + 1) \langle \tilde{L}|_{Z'_F} \rangle^{\dim Z}} \in \widehat{\Pic}(F/k)_{\mathrm{int}, \mathbb{Q}},
\] 
where $Z'$ is the image of $Z$ in $X,$ and $\tilde{L}|_{Z'_F}$ is the image of $\ol{L}$ under 
\[\widehat{\Pic}(X/k) \rightarrow \widehat{\Pic}(Z'/k)_{\mathrm{int}} \rightarrow \Pic(Z_F'/F).
\]
In particular, the denominator is a real number, since $\widehat{\Pic}(Z_F'/F) = \Pic(Z_F'/F)$.

\section{Globally Valued Fields}
We now introduce the notion of globally valued fields. These are fields $K$ with compatible global height functions on $\bb{P}^n(K)$ for each $n \geq 1$. As in \cite{GVFfoundations}, the theory of GVFs is related to both Yuan--Zhang adelic line bundles and Chen--Moriwaki adelic line bundles. The relation with the former allows us to view GVFs as generic points (as suggested by Hultberg) of a family of adelic line bundles, while the latter allows us to still invoke Arakelov theory over this generic point. 

\subsection{Definitions}
Following \cite{GVFfoundations}, we define a globally valued field (GVF) to be a field $K$ with a height function $h: \bigsqcup\limits_{n \geq 1} K^n \to \mathbb{R} \cup \{ -\infty\}$ satisfying:

\begin{enumerate}
\item $h(x) = -\infty$ if and only if $x = (0, \ldots, 0) \in K^n$ for some $n,$ 
\item $h(1,1) = 0,$

\item $h(x) = 0$ for all $x \in K^{\times},$

\item $h(\sigma(x)) = h(x)$ for all $x \in K^n$ and permutation $\sigma \in S_n,$

\item $h(x \otimes y) = h(x) + h(y)$ for all $x \in K^n$ and $y \in K^m$, where $x \otimes y \in K^{mn}$ denotes the Segre product,

\item $h(x) \leq h((x,y))$ for all $x \in K^n$ and $y \in K^m$ where we view $(x,y)$ as an element of $K^{n+m},$

\item There exists $e \geq 0$, called the archimedean error, such that $h(x+y) \leq h((x,y)) + e$ for all $x \in K^n$ and $y \in K^m.$
\end{enumerate}

Using property (3) and (5), it follows that $h$ descends to a function on $h: \mathbb{P}^n(K) \to \mathbb{R} \cup \{-\infty\}$ and hence can be viewed as a Weil height function. By abuse of notation, for $x \in K$ we will let $h(x) = h([x:1])$.

\par Given an adelic curve $(K, (\Omega,\mathcal{A},\nu),\phi)$, we can construct a GVF from it, defining a height function for the field $K$ by the function \[h(a_1, a_2, \ldots, a_n) = \int\limits_{\Omega} \log \max\{ |a_1|_{\omega},\ldots, |a_n|_{\omega}\} d\nu(\omega)\] (see \cite[p.33]{GVFfoundations}). It turns out that we can also go backwards in the cases we are interested in:
\begin{proposition}[Corollary 7.11 in \cite{GVFfoundations}]
Every countable GVF arises from a proper adelic curve.
\end{proposition}
In particular, if we are interested in working with countable GVFs, we can pass to working with proper adelic curve over a countable field (which we will just call a countable proper adelic curve). 
\par Given a GVF $K$, there is an unique GVF structure we may put on any finite extension $L/K$ that is invariant under $\Aut(L/K)$, called the symmetric extension, see \cite[Proposition 10.5]{GVFfoundations}. Note that without imposing invariance under $\Aut(L/K)$, there may be multiple possible GVF structures one can put on $L$, see \cite[Example 10.6]{GVFfoundations}.

\subsection{Ultrafilters of GVFs}
We now recall the notion of an ultrafilter and how this can be applied to GVFs. Ultrafilters of valued fields have been used by the third author to obtain uniformity results, see \cite{LY26, GY26}. We begin with a brief discussion about the general notion of an ultrafilter; the reader interested in a more thorough discussion about ultrafilters can consult \cite{goldbringultrafilters}. 

\begin{defn}[\cite{goldbringultrafilters}, Definitions 1.1.1 and 1.1.8]
Let $I$ be a set. Then, a \textbf{filter} is a family of subsets $\omega$ satisfying the following three properties:
\begin{enumerate}
    \item $\emptyset \not \in \omega,$
    \item If $A \subset B \subset I$ and $A \in \omega,$ then $B \in \omega,$
    \item If $A, B \in \omega,$ then $A \cap B \in \omega$ too.
\end{enumerate}
An \textbf{ultrafilter} is a filter which is maximal; that is, for all $A \subset I,$ either $A \in \omega$ or $I - A \in \omega.$
\end{defn}
It is known that every filter $\omega$ lies inside of some ultrafilter $\omega';$ see \cite[Corollary 1.1.17]{goldbringultrafilters}.
\par We let $\beta I$ be the collection of all ultrafilters on $I;$ this has a topology where a basis of opens is given by $\{\omega \in \beta I| E \in \omega\}$ over all subsets $E \subset I;$ see \cite[Theorem 3.2.1]{goldbringultrafilters}. With this topology, $\beta I$ becomes a compact Hausdorff space, and this topological space satisfies the following universal property: for all compact Hausdorff spaces $K,$ whose underlying set is $|K|,$ we have the equality \[\Hom_{\text{Top}}(\beta I, K) = \Hom_{\text{Set}}(I, |K|);\] see \cite[Theorem 3.2.8]{goldbringultrafilters}.
Given a set map $f: I \rightarrow |K|,$ the corresponding map $\beta f: \beta I \rightarrow K$ is given by sending $\omega$ to the unique point $x \in K$ such that $f^{-1}(U) \in \omega$ for all opens $U$ containing $x.$ 
\par If we view $f$ as a collection of $x_i \in K,$ indexed by $i \in I,$ we can then define the notion of the limit of this collection with respect to an ultrafilter $\omega$ as \[\lim\limits_{\omega} x_i = \beta f(\omega).\]
\par In this paper, we will use this notion on both GVFs and valued fields. First, we have the definition for GVFs, as follows. 
\begin{defn}[\cite{GVFfoundations}, Construction 11.9] \label{defn: ultrafilter}
Suppose that we have a set $I,$ an ultrafilter $\omega,$ and, for each $i \in I,$ a GVF $K_i.$ Then, we define the \textbf{ultraproduct} of the $K_i$ with respect to $\omega$ as the quotient of \[\left\{(x_i)\in \prod\limits_{i \in I} K_i \mid h_{K_i}(x_i) \text{ bounded}\right\}\] by the equivalence relation that $(x_i) \sim (y_i)$ if $x_i = y_i$ on a set $S \in \omega.$
\end{defn}

This ultraproduct is also a GVF, with the addition and multiplication operations done coordinate-wise, and whose height function on a sequence $(x_i)$ is given by $\lim\limits_{\omega} h(x_i).$ Note that this limit exists since the $x_i$ all have bounded height, so this corresponds to a map from $I$ to a compact interval (where this limit exists).
\par We now describe how we can use these constructions geometrically, as discussed in \cite[pp.14-15]{Hul26}.

\begin{proposition} \label{GVFConstruction1}
Let $\omega$ be a non-principal ultrafilter on a set $I$ and let $X$ be a finite type quasi-projective variety over a field $k$ and fix an embedding $X \xhookrightarrow{} \bb{P}_{k}^N$. For each $i \in I$, let $K_i$ be a GVF field mapping to $X$ by $x_i: \Spec K_i \to X.$ After viewing $x_i$ as a point in $\bb{P}^N$, assume that $h(x_i)$ is uniformly bounded. Let $K_{\omega}$ be the ultraproduct of the $K_i$ with respect to $\omega.$ Then there exists a map
$$x_{\omega}: \Spec K_{\omega} \to X$$
induced by the $x_i$'s.
\end{proposition}
\begin{proof}
First, suppose that $\Spec A$ is a finite type affine scheme over $\mathbb{Z}$ with an embedding $\Spec A \rightarrow \mathbb{A}^N_{\mathbb{Z}},$ and we have points $\Spec K_i \rightarrow \Spec A$ where the $K_i$ are GVFs indexed by a set $I,$ such that the coordinates of the points in $\mathbb{A}^N_{\mathbb{Z}}$ have bounded height. That is, there is a constant $C$ such that for all $i, j,$ the image of the $j$th coordinate $T_j$ under the morphism $\mathbb{Z}[T_1, T_2, \ldots, T_N] \rightarrow K_i$ has height at most $C.$ We can then form the ring map \[\mathbb{Z}[T_1, T_2, \ldots, T_N] \rightarrow \prod\limits_{i \in I} K_i,\] whose image lies in the set of sequences whose heights are bounded. We may then further quotient to get a morphism $\mathbb{Z}[T_1, T_2, \ldots, T_N] \rightarrow K_{\omega}.$ Finally, since the image of the ideal cutting out $A$ vanishes in each $K_i,$ it vanishes in $K_{\omega},$ so this induces a morphism $A \rightarrow K_{\omega},$ and thus a point $\Spec K_{\omega} \rightarrow \Spec A.$
\par 
In our situation, we may cover $\bb{P}^N$ with $N+1$ distinct affine opens $\Spec A_j \simeq \mathbb{A}^N_{\mathbb{Z}}.$ Then there exists $j$ such that $I_j = \{i \mid x_i \in \Spec A_j\}$ lies in $\omega.$ We may then pass to this subsequence and deal only with points landing in $\Spec A_j,$ owing to the existence of the map $\prod\limits_{i \in I_j} K_i \rightarrow K_{\omega}$ given by extending the tuple $(f_i)_{i \in I_j}$ with zeroes for all $i \in I - I_j.$ This reduces us to the case of the previous paragraph, giving us the desired result.
\end{proof}

\subsection{Adelic Line Bundles on GVFs}
In this subsection, we relate Yuan--Zhang adelic line bundles and Chen--Moriwaki adelic line bundles, and use these to give another equivalent notion of a GVF structure. The fact that you can move between these bundles has been noted and used in \cite{Hul26}, but for completeness we check these facts explicitly. Throughout this subsection, let $k$ be either $\mathbb{Z}$ or a field.
\par Suppose that $K$ is finitely generated as a field over $k$ (or $\mathbb{Q}$ if $k = \mathbb{Z}$). Recall from \cite[Theorem 7.7]{GVFfoundations} that the data of a GVF structure on a field $K$ is equivalent to the data of a positive functional; that is, a homomorphism 
\[
\text{LDiv}_{\mathbb{Q}}(K) \rightarrow \mathbb{R}
\] 
sending nonnegative lattice divisors to nonnegative numbers, and which vanishes on principal lattice divisors (see \cite[Definition 5.1]{GVFfoundations}). We first show that we can reinterpret this functional as a functional on Yuan--Zhang adelic line bundles, as follows.

\begin{lemma}\label{lem: GVF structure given by map from adelic line bundles}
There is a natural bijection between the set of GVF structures on $K$, which we assume are trivial on $k$ if $k$ is a field, and the set of group homomorphisms \[\widehat{\Pic}(K/k) \rightarrow \mathbb{R}\] such that the composition \[\widehat{\Div}(K/k) \rightarrow \widehat{\Pic}(K/k) \rightarrow \mathbb{R}\] sends (Cauchy sequences of) effective divisors to nonnegative numbers. 
\end{lemma}
\begin{proof}
To prove this lemma, we first show that there is a natural way to map from $\text{LDiv}_{\mathbb{Q}}(K)$ to $\widehat{\Div}(K/k) \otimes_{\mathbb{Z}} \mathbb{Q}.$ We then show that every positive functional extends uniquely to a group homomorphism $\widehat{\Pic}(K/k) \rightarrow \mathbb{R}$ satisfying the condition, and that every such functional on $\widehat{\Pic}(K/k)$ restricts to a positive functional on $\text{LDiv}_{\mathbb{Q}}(K/k).$
\par To give the map $\text{LDiv}_{\mathbb{Q}}(K) \rightarrow \widehat{\Div}(K/k),$ we do casework. Suppose first that $k$ is a field. Then, by \cite[Corollary 8.4]{GVFfoundations}, the quotient of $\text{LDiv}_{\mathbb{Q}}(K)$ by the lattice ideal generated by $\text{LDiv}_{\mathbb{Q}}(k)$ is equal to \[\varinjlim (\Div(X/k) \otimes_{\mathbb{Z}} \mathbb{Q})\] ranging over all projective varieties $X$ equipped with isomorphisms $k(X) \simeq K.$ On the other hand, recall from \cite[\S 2.4.3]{yuan2021adelic} that \[\widehat{\Div(K/k)} = \varinjlim \widehat{\Div}(U/k),\] ranging over all quasiprojective varieties $U$ with isomorphisms $k(U) \simeq K,$ where $\widehat{\Div}(U/k)$ was the completion of $\varinjlim \Div(X, U)$ over projective models $X$ of $U;$ see \cite[\S 2.2.3]{yuan2021adelic} for the definition of $\Div(X, U).$ The morphisms \[\Div(X/k) = \widehat{\Div}(X/k) \rightarrow \widehat{\Div}(K/k) \rightarrow \widehat{\Div}(K/k) \otimes_{\mathbb{Z}} \mathbb{Q}\] for each projective variety $X/k$ with $k(X) \simeq K$ then gives us our desired map $\text{LDiv}_{\mathbb{Q}}(K) \rightarrow \widehat{\Div}(K/k) \otimes_{\mathbb{Z}} \mathbb{Q}.$
\par In the case where $k = \mathbb{Z},$ we instead consider the set $\text{ADiv}_{\mathbb{Q}}(K)$ from \cite{sza23}, which we may do by \cite[Theorem 8.1]{GVFfoundations}. From the definition, this is given as the direct limit of arithmetic divisors $\text{ADiv}_{\mathbb{Q}}(\mathcal{X}),$ which is just $\widehat{\Div}_{\mathbb{Q}}(\mathcal{X})$ as defined in \cite[\S 2.2]{yuan2021adelic}. We can now map $\text{ADiv}_{\mathbb{Q}}(K)$ into $\widehat{\Div}(K/k) \otimes_{\mathbb{Z}} \mathbb{Q},$ since $\mathcal{X}$ is a quasiprojective variety with function field $K.$ Notice that in both of these cases that our map sends principal lattice divisors to principal divisors and positive ones to effective divisors.
\par We now turn to showing that positive functionals on $\text{LDiv}_{\mathbb{Q}}(K)$ extend to positive functionals on $\widehat{\Div}(K/k) \otimes_{\mathbb{Z}} \mathbb{Q}.$ Note in the arithmetic case that, by \cite[Proposition 3.1.16]{sza23}, positive functionals extend uniquely to positive functionals from $\text{ADiv}_{\mathbb{Q}}(K)$ (where positivity for $\text{ADiv}_{\mathbb{Q}}(K)$ refers to the effective cone). In particular, in either case, we have positive functionals $\widehat{\Div}_{\mathbb{Q}}(X/k) \rightarrow \mathbb{R}$ for projective $X/k$ which are compatible with the morphisms induced by functoriality.
\par This means that for every quasi-projective variety $\mathcal{U}$ with function field $K,$ we get a positive functional $l:\widehat{\Div}(\mathcal{U})_{\text{mod}} \rightarrow \mathbb{R}$ (that is, one preserving the order given in \cite[\S 2.4.1]{yuan2021adelic}) using the projection maps $\widehat{\Div}(\mathcal{X}, \mathcal{U}) \rightarrow \widehat{\Div}_{\mathbb{Q}}(\mathcal{X})$ and then taking colimits. Furthermore, these are also compatible with the morphisms arising from functoriality, since these projection maps also commute with the pullback morphisms arising from these diagrams:
\[\begin{tikzcd}
	{\mathcal{U}'} && {\mathcal{X}'} \\
	\\
	{\mathcal{U}} && {\mathcal{X}}
	\arrow[hook, from=1-1, to=1-3]
	\arrow[from=1-1, to=3-1]
	\arrow[from=1-3, to=3-3]
	\arrow[hook, from=3-1, to=3-3]
\end{tikzcd}\]
We now show that this functional uniquely extends to a morphism $\widehat{\Div}(\mathcal{U}) \rightarrow \mathbb{R}.$ To do this, fix a boundary divisor $\overline{\mathcal{E}_0} \in \widehat{\Div}(\mathcal{X}_0, \mathcal{U}).$ Then, recall that $\widehat{\Div}(\mathcal{U})$ is defined by completing with respect to the boundary norm with respect to $\ol{\mathcal{E}_0}.$ 
\par However, for every Cauchy sequence of divisors $\mathcal{D}_i,$ note that there is a sequence of positive rational numbers $\eps_i$ converging to zero, such that \[-\epsilon_i \ol{\mathcal{E}_0} \leq \mathcal{D}_i - \mathcal{D}_{i'} \leq \epsilon_i \ol{\mathcal{E}_0}\] for all $i' \geq i \geq 1.$ Applying $l,$ since $l$ is positive, we get that \[-\epsilon_i l(\ol{\mathcal{E}_0}) \leq l(\mathcal{D}_i) - l(\mathcal{D}_{i'}) \leq \epsilon_i l(\ol{\mathcal{E}_0}).\] But this means that the sequence of $l(\mathcal{D}_i)$ is Cauchy in $\mathbb{R},$ so it therefore converges; we define $l$ of this sequence to be that limit. The same argument shows that this is well-defined, and sends Cauchy sequences of effective divisors to nonnegative numbers. 
\par This condition also uniquely pins down $l.$ To check this, notice that the positivity condition applied to the Cauchy sequences $\pm (\mathcal{D}_i - \mathcal{D}_j) + \epsilon_j \ol{\mathcal{E}_0}$ (for fixed $j$) implies that, if we denote by $\mathcal{D}$ the Cauchy sequence in $\widehat{\Div}(\mathcal{U}),$ \[\pm(l(\mathcal{D}) - l(\mathcal{D}_j)) + \epsilon_j l(\ol{\mathcal{E}_0}) \geq 0,\] or that \[\epsilon_j l(\ol{\mathcal{E}_0}) \geq |l(\mathcal{D}) - l(\mathcal{D}_j)|.\] This implies that $l(\mathcal{D})$ has to be the limit of the $l(\mathcal{D}_j).$
\par The uniqueness and the compatibility of the functionals on $\widehat{\Div}(\mathcal{U}/k)_{\text{mod}}$ imply the functionals $\widehat{\Div}(\mathcal{U}/k) \rightarrow \mathbb{R}$ are compatible, so we get a morphism $\widehat{\Div}(K/k) \rightarrow \mathbb{R}$ (and thus one from $\widehat{\Div}(K/k) \otimes_{\mathbb{Z}} \mathbb{Q}$). Finally, this is trivial on principal divisors by assumption, so this actually induces our desired morphism $\widehat{\Pic}(K/k) \rightarrow \mathbb{R}$ (and using \cite[Proposition 2.5.1]{yuan2021adelic}).
\par For the other direction, given our morphism $\widehat{\Pic}(K/k) \rightarrow \mathbb{R},$ we get a functional on lattice divisors by \[\text{LDiv}_{\mathbb{Q}}(K) \rightarrow \varinjlim (\widehat{\Div}(X/k) \otimes_{\mathbb{Z}} \mathbb{Q}) \rightarrow \widehat{\Div}(K/k) \otimes_{\mathbb{Z}} \mathbb{Q} \rightarrow \widehat{\Pic}(K/k) \otimes_{\mathbb{Z}} \mathbb{Q} \rightarrow \mathbb{R}.\] This sends principal divisors to zero, and the positivity condition implies, in particular, that positive lattice divisors are sent to nonnegative numbers, which finishes the proof of the lemma.
\end{proof}

Let $K$ be a countable GVF and let $X/K$ be a projective variety. Following Hultberg \cite[Section 2.2]{Hul26}, we say a model of $X/K$ is the data of a finite type quasi-projective scheme $\mathcal{Y}$ over $\Spec \bb{Z}$, and a flat projective morphism $\mathcal{X} \to \mathcal{Y}$, along with a map $\Spec K$ to the generic point of $\mathcal{Y}$ and an isomorphism $X \simeq \mathcal{X}_K$. An adelic line bundle $\ovl{L}$ on $X/K$ is any adelic line bundle $\ovl{\mathcal{L}}$ on $\mathcal{X}$, in the sense of Yuan--Zhang, where $\mathcal{X} \to \mathcal{Y}$ is a model. Here, if $K$ is characteristic $0$ and is an archimedean GVF, then $\ovl{\mathcal{L}}$ is an element of $\wh{\Pic}(\mathcal{X}/\bb{Z})$. If $K$ is a characteristic $0$ non-archimedean GVF, we only consider models $(\mathcal{X} \to \mathcal{Y})$ with $\mathcal{Y}$ quasi-projective over $\bb{Q}$ and we let $\ovl{\mathcal{L}}$ be an element of $\wh{\Pic}(\mathcal{X}/\bb{Q})$. Finally, if $K$ is characteristic $p$, then it is an element of $\wh{\Pic}(\mathcal{X}/\bb{F}_p)$. Any two models are dominated by a third and we identify adelic line bundles if they agree on a dominating model. We say that $\ovl{L}$ is strongly nef, nef or integrable if $\ovl{\mathcal{L}}$ is on $\mathcal{X}$. 
\par 
Given an adelic line bundle $(\mathcal{X} \to \mathcal{Y}, \mathcal{L})$, we may pass to the analytification and obtain a metrized line bundle $\mathcal{L}^{\an}$ on $\mathcal{X}^{\an} \to \mathcal{Y}^{\an}$. Here, we take the Berkovich analytification over $(\bb{Z}, | \cdot |_{\infty})$ \cite[Section 3.1]{yuan2021adelic}. The algebraic morphism $\iota: \Spec K \to \mathcal{Y}$ gives us $(\Spec K)^{\an} \to \mathcal{Y}^{\an}$. We now view $K$ as an adelic curve $(K,(\Omega, \mathcal{A},\nu),\phi)$. Pulling back, we obtain for each $\omega \in \Omega$ a metrized line bundle $| \cdot |_{\omega}$ on the line bundle $L = \iota^* \mathcal{L}$ of $X_{\omega}^{\an}$ over $K_{\omega}$. As remarked by Hultberg in the discussion above \cite[Remark 2.2.2]{Hul26}, these metrics are compatible and gives rise to a Chen--Moriwaki adelic line bundle on $X/K$. We will sketch a proof below,

\begin{proposition} \label{YuanZhangChenMoriwai1}
The data of the metrics $\{| \cdot |_{\omega}\}_{\omega \in \Omega}$ gives $L$ the structure of a Chen--Moriwaki adelic line bundle on $X/K$.
\end{proposition}

\begin{proof}
First by \cite[Remark 3.9]{DHS24}, we know that if $\ovl{L}$ is a lattice adelic line bundle then its analytification is a CM adelic line bundle; see \cite[Definition 3.8]{DHS24} for the definition of a lattice adelic line bundle. Now, suppose that $\ovl{L}$ on $X/K$ is induced by an adelic line bundle on the model $\mathcal{X} \to \mathcal{Y}$. If $\ovl{L}$ comes from a model adelic line bundle $\ovl{\mathcal{L}}$ of a compactification $\ovl{\mathcal{X}}$ of $\mathcal{X}$, it is easy to verify that $\ovl{L}$ is a lattice adelic line bundle in the geometric case, i.e. when $K$ is nonarchimedean. For $k = \bb{Z}$, because of the archimedean places, we can only approximate $\ovl{L}$ by lattice adelic line bundles. However, one can show that $\ovl{L}$ can be uniformly approximated by lattice adelic line bundles (for instance, see the proof of \cite[Proposition 3.1.10]{sza23}). We hence obtain that model adelic line bundles are CM adelic line bundles.
\par In general, suppose we have an adelic line bundle $\ovl{\mathcal{L}}$ on $\pi: \mathcal{X} \to \mathcal{Y}.$ Fix a projective compactification $\pi: \ovl{\mathcal{X}} \to \ovl{\mathcal{Y}}$ so that $\ovl{\mathcal{Y}} \setminus \mathcal{Y} = D$ is Cartier. We then endow $O(D)$ with a model adelic metric $\ovl{O(D)}$, with archimedean metrics chosen (if necessary) so that pulling back to $\Spec K$ gives us a lattice adelic line bundle. Then $\ovl{N} = \pi^* \ovl{O(D)}$ is an adelic line bundle and has a Green function $g$, with logarithmic poles on $\pi^*D$, satisfying $g(x) = h(\pi(x))$, where $h$ is the Green function for $\ovl{O(D)}.$
\par Now let $\{\ol{\mathcal{L}_i}\}_{i \geq 1}$ be our sequence of model adelic line bundles approximating $\ovl{\mathcal{L}}$. For $\omega \in \Omega$, let $\omega_y \in \mathcal{Y}^{\an}$ be its image coming from $\Spec K \to \mathcal{Y}$. Then by the definition of the boundary topology, we have a sequence of positive constants $\eps_n \to 0 $ so that for $\omega \in \Omega$, we have
$$d_{\omega}( \ovl{\mathcal{L}}_n, \ovl{\mathcal{L}}) \leq \eps_n \sup_{x \in X_{\omega_y}^{\an}} g(x) = \eps_n h(\omega_y).$$
As $\ovl{O(D)}$ is a lattice adelic line bundle, it an integrable CM adelic line bundle on $K$ itself by \cite[Lemma 3.7 and Remark 3.9]{DHS24}. We hence obtain that the metrics of $\ovl{\mathcal{L}}$ is dominated, since the metrics for $\ol{\mcL_n}$ are. Measurability follows from the fact that pointwise limit of measurable family of metrics are still measurable by \cite[Proposition 6.1.29]{CM20}, and hence $\ovl{\mathcal{L}}$ induces a CM adelic line bundle as desired.     
\end{proof}

If $\ovl{L}$ is our adelic line bundle on $X/K$ where $K$ is a GVF, we will denote the induced CM adelic line bundle by $\ovl{L}^{\CM}$. If $\ovl{L}$ is nef, then $\ovl{L}^{\CM}$ has semipositive metrics as a CM adelic line bundle, i.e. $c_1(\ovl{L}_{\omega}^{\CM}) \geq 0$ for all $\omega \in \Omega$ (for instance, see \cite[Remark 2.3.18]{CM20} for model adelic line bundles, and \cite[Proposition 2.3.3]{CM20} to obtain the general case). Hence if $\ovl{L}$ is integrable, then $\ovl{L}^{\CM}$ is also integrable in the sense of Chen--Moriwaki.

Let $\dim X = n$. Given $n+1$ adelic line bundles $\ovl{L}_1,\ldots,\ovl{L}_{n+1}$ on $X/K$, we now have two viable ways to compute a global intersection number. The first is to view them as Yuan--Zhang adelic line bundles (which we may do after passing to a finitely generated subfield of $K$ over which these bundles are defined), take their Deligne pairing and then evaluate the height functional coming from the GVF structure on the Deligne pairing. On the other hand, we may consider them as CM adelic line bundles $\ovl{L}_1^{\CM},\ldots, \ovl{L}_{n+1}^{\CM}$ and then take their global intersection number as defined by Chen--Moriwaki. It turns out that these both agree.  
\begin{lemma}[Proposition 2.5.5 from \cite{Hul26}]\label{lem: comparing intersections for YZ and CM}
Let $X$ be a projective variety over $K$ of dimension $n,$ and let $\ol{L}_1, \ol{L}_2, \ldots, \ol{L}_{n+1}$ be integrable Yuan-Zhang adelic line bundles. Then, their intersection number as defined in \cite[Definition 2.2.1]{Hul26} equals their intersection number as Chen-Moriwaki adelic line bundles.  
\end{lemma}

\subsection{Equidistribution of Small Points}
Let $K$ be a countable GVF, say algebraically closed. Let $X$ be a projective variety over $K$ and first view $K$ as an adelic curve. Let $\ovl{L}$ be a semipositive adelic line bundle such that $L$ is geometrically semiample and big. Then Chen--Moriwaki in \cite[\S 8.11]{chen2024positivity} establish an equidistribution theorem for small points on $X(\ovl{K})$. 
\par Equidistribution of small points goes back to Szpiro--Ullmo--Zhang, where in \cite{SUZ97} they proved it for abelian varieties and archimedean places. For non-archimedean places, a candidate probability measure that lives on the Berkovich analytification was proposed by Chambert-Loir in \cite{CL06}. Gubler extended Szpiro--Ullmo--Zhang's result to a tropical equidistribution theorem over non-archimedean places in \cite{gubler2007bogomolov, Gub08b}. Yuan in \cite{Yua08} subsequently generalized the previous results to the setting of big line bundles that holds for both archimedean and non-archimedean places.
\par 
To carry out the Szpiro--Ullmo--Zhang strategy, we want to show that if we have another semipositive adelic line bundle $\ovl{M}$, which is geometrically semiample and big, such that
$$h_{\ovl{L}}(X) = h_{\ovl{M}}(X) = 0$$
and both $\ovl{L},\ovl{M}$ are nef, then we can obtain an equality of measures in an appropriate setting.
\par 
We first need the following proposition about existence of small points over a suitable GVF extension of $K$. In the canonical height case, one need not take a GVF extension at all. 

\begin{proposition} \label{SmallPointsExistence1}
Let $K$ be a countable GVF, $X/K$ a projective variety and $\ovl{L}$ a nef semipositive adelic line bundle which is geometrically big. Assume that $h_{\ovl{L}}(X) = 0$. Fix an nef semipositive adelic line bundle $\ovl{M}$ which is geometrically ample. Then there exists a countable GVF extension $K'/K$ such that for each $n \geq 1$, if $\ovl{L}_n = \ovl{L} + \frac{1}{n} \ovl{M}$ and $\ovl{L}_0 = \ovl{L}$, then $X(K')$ has a generic sequence of small points for $\ovl{L}_n$.     
\end{proposition}

\begin{proof}
We let $\ovl{M}$ be a nef semipositive adelic line bundle which is geometrically ample on $X$. We consider $\ovl{L}_t = \ovl{L} + t \ovl{M}$. Then for $t > 0,$ $\ovl{L}_t$ is geometrically ample and we may apply \cite[Proposition 2.6.8]{Hul26} so that the essential minimum of $\ovl{L}_t$ is $\leq h_{\ovl{L}_t}(X) = O(t)$. Furthermore, $h_{\ovl{L}_t}(x) \geq h_{\ovl{L}}(x)$. As observed by \cite[Remark 2.6.3]{Hul26}, we can find a GVF extension $K_t'$ of $K$ such that there is a Zariski dense set $T_t$ of points of $X(K_t')$ attaining the essential minimum for $\ovl{L}_t$. We may replace $T_t$ with a countable subset while preserving Zariski density as $X$ is of finite type over $K$. 
\par 
We now consider $T' = \bigcup_{n \in \bb{N}} T_{1/n}$ and we let $K'$ be a countable extension of $K$ so that it contains all the coordinates of $T'$. This gives us our generic sequence for $\ovl{L}_n$ for $n \geq 1$. Then enumerating $T'$, we obtain a generic sequence $(x_n)$ such that $h_{\ovl{L}}(x_n) \to 0$ as desired, giving us a sequence for $\ovl{L}_0$.
\end{proof}

We now have the following theorem. Recall that a line bundle $L$ on a projective variety $X/K$ is said to be semiample if there exists $m \geq 1$ such that $L^{\otimes m}$ is globally generated (see \cite[Section 2.I.B]{Laz04}). It is said to be big if $\dim H^0(X,L^m) \geq c m^{\dim X}$, for some $c > 0$ and for all sufficiently large $m$ (see \cite[Section 2.2]{Laz04}). Semiampleness is preserved under pullback by any morphism, while bigness is preserved under pullback by surjective generically finite morphisms (see \cite[Lemma 2.1.13]{Laz04}, for instance).

\begin{theorem} \label{EqualityMeasures1}
Let $K$ be a countable GVF and $X$ a projective variety over $K$. View $K$ as an adelic curve and let $\ovl{L},\ovl{M}$ be two nef semipositive adelic line bundles which are geometrically semiample and big. Assume that there exists $ a > 0$ such that $a \ovl{L} - \ovl{M}$ is nef, and that
$$h_{\ovl{L}}(X) = h_{\ovl{M}}(X) = 0.$$
Then there exists a countable GVF extension $K'/K$, such that after viewing $K' = (K', (\Omega,\mathcal{A},\nu),\phi)$ as an adelic curve, we have 
$$\frac{c_1(\ovl{L})_{v}^{\dim X}}{\deg_L(X)} = \frac{c_1(\ovl{M})_{v}^{\dim X}}{\deg_M(X)}$$
for all $v \in \Omega$ outside a set of measure zero.
\end{theorem}

\begin{proof}
By Proposition \ref{SmallPointsExistence1}, we find a countable extension $K'/K$, algebraically closed, such that $X(K')$ has a generic sequence $(x_n)$ of small points for $\ovl{L}$. Since $a \ovl{L} - \ovl{M}$ is nef, we know that $a h_{\ovl{L}} \geq h_{\ovl{M}}.$ Hence the same sequence is also small for $\ovl{M}$. 
\par 
Viewing $x_n$ as a sequence of closed subschemes of $X/K'$, the statement of \cite[Theorem F, \S 1.2.4]{chen2024positivity} tells us that these points equidistribute to a global family of Borel measures given by the Monge-Amp\`ere measures (see \cite[Definition 7.5.3]{chen2024positivity}). However, this convergence of measures is over the global adelic space $X_{K'}^{\an}$ which is the disjoint union of $X_{v}^{\an}$ for $v \in \Omega$. To obtain the equidistribution on an individual fiber $X_{\omega}^{\an}$, we apply \cite[Proposition 7.7.1]{chen2024positivity}, which tells us that, for all places $v$ on our adelic curve $K'$ outside of a set of measure zero, the small height points $x_n$ equidistribute to the Monge-Amp\`ere measure $\frac{c_1(\ovl{L})_v^{\dim X}}{\deg_L(X)}$, where we normalize it so that it is a probability measure. As $(x_n)$ is a generic sequence of small points for $\ovl{M}$ too, we obtain the equality
$$\frac{c_1(\ovl{L})_v^{\dim X}}{\deg_L(X)} = \frac{c_1(\ovl{M})_v^{\dim X}}{\deg_M(X)}$$
for $v \in \Omega$ outside a set of measure zero, as desired.
\end{proof}

Theorem \ref{EqualityMeasures1} allows us to run the Szpiro--Ullmo--Zhang strategy over adelic curves. In the classical situation where $K$ is a number field or a global function field, the semiampleness assumption is not needed (see \cite{Yua08}). However, Chen--Moriwaki's equidistribution result requires this in order to control some slope estimates as their Chen--Moriwaki adelic line bundles are more general. In applications to Bogomolov-type problems, the semiampleness assumption is very mild as it is preserved under general pullbacks.

\subsection{GVF Analytification}
Let $K$ be a countable GVF and let $X$ be a finite type scheme over $K$. We describe a construction (following \cite[Section 2.2]{DHS24}) called the GVF analytification, which mimics the Berkovich analytification for the valued field case. We will denote this space by $X_{\GVF,K}$, or just $X_{\GVF}$. 
\par 
As a set, $X_{\GVF}$ consists of points $(\pi(x), h_x)$, where $\pi(x)$ is a point on the scheme $X$, and $h_x$ is a GVF structure on the residue field $k(\pi(x))$ that extends the GVF structure on $K$. Let $U$ be an open of $X$ containing $\pi(x)$. Given functions $f_1,\ldots,f_n \in O(U)$, we may consider passing $f_1,\ldots,f_n$ to elements of $k(\pi(x))$, and then evaluating our height function $h_x$ on the tuple $(f_1,\ldots,f_n)$, This defines for us a function 
$$(f_1,\ldots,f_n): U_{\GVF} \to \bb{R}_{\geq 0} \text{ given by } y \mapsto h_y(f_1,\ldots,f_n).$$
We endow $X_{\GVF}$ with the weakest topology such that the map $\pi: X_{\GVF} \to X$, given by taking $\pi(x)$, is continuous for the constructible topology on $X$, and for every Zariski open $U$ and tuple $(f_1,\ldots,f_n)$, the function $(f_1,\ldots,f_n)$ is continuous. This topology is Hausdorff and locally compact \cite[Remark 2.7]{DHS24}. 
\par Given an adelic line bundle $\ovl{L}$ on $X/K$, we may use it to define a function $h_{\ovl{L}}: X_{\GVF} \to \bb{R}$, given by evaluating $\ovl{L}$ on the GVF point $x \in X_{\GVF}$. The below proposition is stated in \cite[Section 2.5]{Hul26}. Since we cannot find a reference for this, we briefly sketch a proof.

\begin{proposition} \label{GVFAnalytification1}
Let $\ovl{L}$ be an adelic line bundle on $X/K$. Then the height functional 
$$h_{\ovl{L}} : X_{\GVF} \to \bb{R} \text{ given by } x \mapsto h_x( \ovl{L} \vert_{k(\pi(x))})$$
is continuous. 
\end{proposition}

\begin{proof} 
First by \cite[Proposition 3.18]{DHS24}, we know that it is continuous if $\ovl{L}$ is a lattice adelic line bundle, see \cite[Definition 3.8]{DHS24} for the definition. Now, let $\ovl{L}$ on $X/K$ be induced by the model $\mathcal{X} \to \mathcal{Y}$. Then as in the proof of Proposition \ref{YuanZhangChenMoriwai1}, if $\ovl{L}$ comes from a model Yuan--Zhang adelic line bundle $\ovl{\mathcal{L}}$ of a compactification $\ovl{\mathcal{X}}$ of $\mathcal{X}$, then $\ovl{L}$ is a lattice adelic line bundle if $K$ is of positive characteristic and $\ovl{L}$ can be uniformly approximated by lattice adelic line bundles if $K$ is characteristic zero. We hence obtain continuity for model Yuan--Zhang adelic line bundles. 
\par 
In general, we fix a model adelic line bundle $\ovl{N}$ whose underlying divisor is $\ovl{\mathcal{X}} \setminus \mathcal{X}$ such that if $s$ is the section of $N$ representing $\ovl{\mathcal{X}} \setminus \mathcal{X}$, then $\wh{\divv}(s)$ is a strictly effective Cartier divisor. Then if $\{\ol{\mathcal{L}_i}\}_{i \geq 1}$ is our sequence of model adelic line bundles approximating $\ovl{\mathcal{L}}$, we have a sequence of positive constants $\eps_n \to 0 $ so that 
\begin{equation} \label{eq: GVFHeightComparison1}
\left| h_{\ovl{\mathcal{L}}}(x) - h_{\ovl{\mathcal{L}}_n}(x) \right| \leq \eps_n h_{\ovl{N}}(x).
\end{equation}
Indeed as we have 
$$-\eps_n \log |s(x) |_{\ovl{N}} \leq \log | \cdot |_{\ovl{\mathcal{L}}} - \log | \cdot |_{\ovl{\mathcal{L}}_n} \leq \eps_n \log |s(x)|_{\ovl{N}} \text{ for all } x \in X^{\an},$$
viewing $K$ as an adelic curve and using the formula for heights gives us (\ref{eq: GVFHeightComparison1}). 
\par 
As $h_{\ovl{N}}$ is already known to be continuous, the set $\{h_{\ovl{N}}(x) < M\}$ is open and we have uniform convergence on such an open set. As $M \to \infty$, this exhausts $X_{\GVF}$ and hence gives local uniform convergence. This gives us $h_{\ovl{\mathcal{L}}} = h_{\ovl{L}}$ is continuous as desired.
\end{proof} 
\par 
There is a natural analytification map $X(K) \to X_{\GVF}$. We now say that a GVF $K$ is existentially closed if for any such finite type scheme $X$, the analytification map $X(K) \to X_{\GVF}$ has dense image. By \cite[Theorem 2.5.3]{Hul26}, following \cite{sza23} and \cite{BH22}, the globally valued fields $\ovl{\bb{Q}}[1]$ and $\ovl{k(t)}[1]$ are existentially closed, where we use the standard global Weil heights on $\ovl{\bb{Q}}$ and $\ovl{k(t)}$ for our height functions.

\section{Non-Archimedean Canonical Measures for Valued Fields}\label{sec: non-archimedean canonical measures}

Our goal in this section is to show that the strategy of \cite{gubler2007bogomolov} and \cite{yamaki2013geometric} to prove the Bogomolov conjecture in the case where we have a place of bad reduction carries over from the discretely valued setting to the general (height one) valued setting. Ideally, we would like to apply these arguments to the limiting abelian variety $A_{\omega}/K'_{\omega}$ (after completing at a place $v$ of bad reduction) in order to conclude. However, $(K'_{\omega}, v)$ will no longer be a discretely valued field.
\par Let $A$ be an abelian variety over an algebraically closed nonarchimedean (non-trivally) valued field $K$ with valuation $v.$ As in \cite{yamaki2013geometric}, we let $b(A)$ be the dimension of the abelian part of the reduction of $A$ at $v.$ Given a subvariety $X \subset A,$ we can make sense of the abelian part of $X$ as follows. For a Mumford model $\mathcal{A}$ of $A$ over $K,$ we can write it as a quotient $p: \mathcal{E} \rightarrow \mathcal{A}$ of a semiabelian formal scheme, which itself has another quotient $q: \mathcal{E} \rightarrow \mathcal{B}$ to an abelian formal scheme. Let $\mathcal{X}$ be the Zariski closure of $X$ in $\mathcal{A},$ and $\mathcal{Y}$ a quasicompact open subscheme of $\mathcal{E}$ such that $p(\mathcal{Y}) = \mathcal{X}.$ Then, we let $b(X)$ be the dimension of $q(\mathcal{Y});$ this is well-defined by \cite[Lemma 4.2]{yamaki2013geometric}.
\par The statement that Yamaki proves that we wish to generalize is the following.  
\begin{thm}[Theorems 5.3, 5.4 from \cite{yamaki2013geometric}]\label{thm: yamaki result}
    Let $S$ be an irreducible normal projective variety over an algebraically closed field, and let $F$ be its field of fractions. Let $A$ be an abelian variety over $\ol{F},$ and endow $\ol{F}$ with the function field height as in \cite[\S 0.5]{yamaki2013geometric}. Let $X$ be a subvariety of $A,$ and $G_X \subset A$ the stabilizer of $X.$ Suppose that for some place $v$ of $\ol{F}$ arising from a codimension one point of the normalization in some finite extension of $F$, either:
    \begin{enumerate}[(i)]
        \item $\dim (X/G_X) > b((X/G_X)_v),$ or
        \item $\dim (X/G_X) \geq b((A/G_X)_v)$ and $X$ is not a special subvariety (see \cite[Definition 2.5]{yamaki2013geometric}).
    \end{enumerate}
    Then, $X$ does not have a Zariski dense set of small points.  
\end{thm}
\begin{rmk}
    In the case that $b(A_v) = 0$ for some place $v$ (i.e.~$A$ has totally degenerate bad reduction at $v$), this result is due to \cite{gubler2007bogomolov}.
\end{rmk}
\begin{rmk}
    Note that if $X$ is an ample divisor in $A$, then its stabilizer $G_X$ is finite, so $\dim (X/G_X) = \dim X.$ Furthermore, $\dim X \geq b(A_v)$ if $v$ is a place of bad reduction, since $b(A_v) \leq \dim A-1 = \dim X.$
\end{rmk}
\par The main thing we need to check is that one can remove the assumption that $(K, v)$ is a discretely valued field through the proofs of \cite[Theorems 5.3, 5.4]{yamaki2013geometric}. We can do this because of the recent work of Gubler and Stadl\"oder, \cite{gubler2023monge}, which proves a lot of the results about the Chambert-Loir measures discussed in \cite{gubler2010non} but without the discrete valuation assumption.
\par One of the key structural results used by Yamaki for Theorem \ref{thm: yamaki result}(i) is \cite[Theorem 1.1]{gubler2010non}, which states that one can express the tropicalization of a closed subvariety $X$ of $A$ as a union of certain canonical simplices, and furthermore the Monge-Amp\`ere measure associated with an ample line bundle $L$ can be expressed as sums of the Lebesgue measure supported on these simplices. The result cited only applies to discretely valued fields; however, we may replace this result with \cite[Corollary 1.4]{gubler2023monge} and the discussion in the proof of \cite[Proposition 6.3]{gubler2023monge} to get the analogous result for our more general (non-discretely valued) situation, which for completeness we state below.
\begin{thm}\label{thm: measure sum of polytope measures}
Let $(K,v)$ be a complete nonarchimedean valued field which is algebraically closed. Let $A$ be an abelian variety over $K,$ $X$ a closed subvariety of $A,$ and let $L$ be an ample line bundle. Then, the tropicalization of the Monge-Amp\`ere measure of $L,$ which we denote $\mu^{\text{trop}}_{X_v, L},$ can be expressed as \[\sum r_j\mu_{\Delta_j},\] where the $\Delta_j \subset \mathbb{R}^n$ are polysimplices and $\mu_{\Delta_j}$ is the Lebesgue measure supported on the polysimplex $\Delta_j.$
\end{thm}
Yamaki's proof of Theorem \ref{thm: yamaki result}(i) also requires that the following bound on the dimension of polysimplices holds, which we state in our more general setting.
\begin{proposition}\label{prop: support of measure and dimension of non-degenerate}
In the setting of the previous theorem, the union of these polysimplices is the tropicalization of $X,$ and we can take them to be precisely the non-degenerate strata (as in \cite[\S 6.3]{gubler2010non}) of a strictly polystable alteration of $X.$ In particular, the dimension of these polysimplices all lie between $\dim X - b(X_v)$ and $\dim X.$
\end{proposition}
\begin{remark}
In the case where $K$ is the algebraic closure of a discretely valued field, this is given by \cite[Theorem 1.1]{gubler2010non} and \cite[(4.6.14)]{yamaki2013geometric}.
\end{remark}
Note that these strictly polystable alterations exist by the polystable alteration theorem \cite[Theorem 5.2.19]{adiprasito2018log}, for instance, which proves the existence of an alteration $\mscr X' \to \mscr X$ with $\mscr X'$ strictly polystable and with smooth generic fiber (that is, strongly nondegenerate; see \cite[\S 4.4]{berkovich2004smooth}). 
\begin{proof}[Proof of Proposition \ref{prop: support of measure and dimension of non-degenerate}]
The fact that these polysimplices correspond to the non-degenerate strata follows from the proof of \cite[Proposition 6.3]{gubler2023monge}, which proves that the Monge-Amp\`ere measure is a sum of the Lebesgue measures supported precisely on the polysimplices that are non-degenerate in the skeleton of a polystable alteration $\mathcal{X}'$. Furthermore, these remain polysimplices when we pass to the tropicalization, since by \cite[Theorem 7.8]{gubler2023monge} the tropicalization restricted to each of these polysimplices is a linear isomorphism.
\par Then, the fact that the union of the simplices is equal to the entire tropicalization of $X$ follows by the arguments outlined in \cite[Remark 7.2]{gubler2023monge}. For the dimension bound, we use the definition of non-degeneracy, where the dimension of the strata is preserved when passing to the abelian part, and the dimension of the corresponding polysimplex is preserved under tropicalization; see \cite[Equation (6.2.1)]{gubler2023monge}. Combining this with the fact that for a strata $S$ in the special fiber of the skeleton $\mathcal{X}'$ with corresponding canonical face $\Delta_S,$ $\dim \Delta_S = \dim X - \dim S,$ yields the desired result (see the discussion in \cite[\S 5.4]{gubler2023monge}, for instance).
\end{proof}
The remainder of Yamaki's proof of Theorem \ref{thm: yamaki result}(i) modulo the equidistribution result carries over, replacing ``simplex" with ``polysimplex." Indeed, \cite[Proposition 4.5]{yamaki2013geometric} holds in our more general setting as well, since the facts they use only rely on results in \cite[\S 3]{gubler2007tropical}, which do not require the discrete valuation assumption.
\par We now turn to Theorem \ref{thm: yamaki result}(ii). For this case, Yamaki makes use of Gubler's result \cite[Corollary A.2]{yamaki2013geometric}, which states that the lower bound given in Proposition \ref{prop: support of measure and dimension of non-degenerate} is attained. In the proof of this corollary (which in turn is a corollary of \cite[Theorem A.1]{yamaki2013geometric}), we can use \cite[Corollary 1.4]{GS23} to replace \cite[Theorem 6.7]{gubler2010non} (since we are only concerned with studying $c_1(\ol{L}|_X)^{\wedge d}$) as well as the same replacements we used in the proof of Proposition \ref{prop: support of measure and dimension of non-degenerate}. This leaves us with proving Theorem A.5 in the same appendix. Thus, in order to show that \cite[Theorem 5.4]{yamaki2013geometric} applies in the more general valued field setting (assuming equidistribution), it suffices to check that \cite[Theorem A.5]{yamaki2013geometric} holds in this more general setting. That is, we need to show the following.
\begin{thm}\label{thm: polystable version of Theorem A5}
Let $\mathcal{A}_0$ be a Mumford model of $A$ over $K^{\circ},$ the valuation ring of our complete and algebraically closed nonarchimedean valued field $K.$ Let $\mathcal{X}_0$ be the Zariski closure of $X$ in $\mathcal{A}_0,$ and let $\varphi: \mathcal{X}' \rightarrow \mathcal{X}_0$ be a strongly non-degenerate strictly polystable alteration. Then, there is a $b(X)$-dimensional stratum of the special fiber of $\mathcal{X}'$ whose corresponding polysimplex is non-degenerate.
\end{thm}
Note that the use of semistable in \cite[Theorem A.5]{yamaki2013geometric} suffices for their purposes because the appendix starts with $X \subset A$ defined over a field with a discrete valuation, to which de Jong's semistable alteration theorem applies. However, in the case that $X$ is defined over a field with an arbitrary height one valuation rather than a discrete valuation, we only have the polystable alteration theorem of \cite[Theorem 5.2.19]{adiprasito2018log}, so we must show the theorem in the setting of strongly non-degenerate strictly polystable alterations.
\begin{proof}
The proof of \cite[Theorem A.5]{yamaki2013geometric} occurs after passage to the completion of the algebraic closure, and so the only potential places where the proof uses the discretely valued field hypothesis is in assuming the alteration is semistable. These arise in citations of \S 4-5 of \cite{gubler2010non} throughout the proof. But the results \S 4 of \cite{gubler2010non} work in the more general setting, so it is just the results of \S 5 of \textit{ibid.}~that we need to generalize from strictly semistable to the case of strongly non-degenerate strictly polystable schemes. However, all of the results in \cite[\S 5]{gubler2010non} carry over to strongly non-degenerate strictly polystable schemes by \cite[Remark 5.19]{gubler2010non}. Hence, one can follow the proof word-for-word in \cite[Theorem A.5]{yamaki2013geometric} to prove the above theorem.
\end{proof}

To obtain our generalization of Theorem \ref{thm: yamaki result} (\cite[Theorem 5.4]{yamaki2013geometric}), we also need to generalize \cite[Lemma 2.12]{yamaki2013geometric}, which for us takes the following form.

\begin{lemma}\label{lem: dimension zero bogomolov}
Let $A$ be an abelian variety over an algebraically closed nonarchimedean GVF $K.$ Then, for any irreducible closed subvariety $X \subset A$ such that $\dim X/G_X = 0,$ $X$ can only have a dense set of small points if it is special. 
\end{lemma}
\begin{rmk}
The notion of a special subvariety from \cite[Definition 2.5]{yamaki2013geometric} extends to this setting, where we let $k$ be the height zero subfield of $K$ and we allow the point $x$ to have canonical height zero in general.
\end{rmk}
\begin{proof}
The only part of the proof of this lemma which doesn't immediately generalize is the claim, which is part of \cite[Proposition 2.11]{yamaki2013geometric}, that $X/G_X$ being special holds if and only if $X$ is. Furthermore, to prove that claim along the lines of Yamaki's proof, it suffices to check the following if $\phi: A \rightarrow B$ is a homomorphism of abelian varieties, then the map sends points of canonical height zero on $A$ to points of canonical height zero on $B,$ and furthermore that this restricted map is surjective. For the first fact, this follows since for any ample symmetric line bundles $L, M$ on $A, B,$ respectively, there is some integer $n > 0$ such that $L^{\otimes n} \otimes \phi^*M^{-1}$ is ample, meaning that $n \hat{h}_L(x) \geq \hat{h}_M(\phi(x)) \geq 0$ for all $x \in A;$ in particular, if $x$ is canonical height zero, $\phi(x)$ is too.
\par For the surjectivity statement, recall that there exists an abelian variety $C/K$ such that $B \times C$ is isogeneous to $A,$ where the homomorphism $A \rightarrow B \times C$ composed with the projection to $B$ is just $\phi.$ Now, composing with its dual, we get a map $[n]: B \times C \rightarrow B \times C.$ Finally, let $y \in B(K)$ be a height zero point. Then, observe that the canonical height of any pre-image $y'$ of $y$ is zero, since $\hat{h}(y) = n^2\hat{h}(y') = 0.$ So then taking the image of $(y', z)$ under $B \times C \rightarrow A$ where $z$ is any height zero point of $C,$ we find that its image in $A$ is a height zero point whose image under $\phi$ is $y,$ as desired.  
\end{proof}

As a consequence of the above, by applying Theorem \ref{EqualityMeasures1} in place of the equidistribution argument and following Yamaki's proof of Theorem \ref{thm: yamaki result} (\cite[Theorem 5.4]{yamaki2013geometric}), we obtain the following. 

\begin{theorem} \label{YamakiGVF1}
Let $K$ be a nonarchimedean countable GVF with height zero subfield $k,$ and let $A$ be an abelian variety over $K$. Assume that $\Tr_{\ovl{K}/\ovl{k}}(A) = 0$. Let $X$ be a subvariety of $A$ and $G_X$ the stabilizer of $X$. Assume that $X$ is not of the form $x+B$, where $x$ is a point of canonical height zero and $B$ is an abelian subvariety. Let 
$$\Omega' = \{ v \in \Omega \mid A \text{ has bad reduction at }v\}$$
and assume for a positive measure of places of $\Omega'$, we have $\dim(X/G_X) \geq b((A/G_X)_v).$ Then $X$ does not have a dense set of small points.
\end{theorem}

\section{Bogomolov Conjecture for Divisors over GVFs} \label{sec: DivisorGVF}
Let $K$ be a nonarchimedean GVF, which we may assume to be algebraically closed and let $k$ be its height zero subfield. Let $A$ be an abelian variety over $K$. Assume that $A$ is traceless, i.e. $\Tr_{K/k}(A) = 0$. We aim to prove the Bogomolov conjecture for divisors $D$ in this setting. 

\begin{theorem} \label{GVFDivisor1}
Let $K$ be an algebraically closed non-archimedean GVF and $k$ its height zero subfield. Let $A/K$ be a $K/k$-traceless abelian variety and let $D$ be a geometrically integral divisor in $A$. Let $L$ be a symmetric ample line bundle on $A$. Then 
$$\h_{A,L}(D) = 0 \implies D = x +B$$
where $B \subseteq A$ is an abelian subvariety. 
\end{theorem}

Let $k_0 = \bb{Q}$ if $K$ is of characteristic $0$ and $\bb{F}_p$ if $K$ is of characteristic $p$. Fix some $n \geq 3$ not divisible by the characteristic of $K$. Then after choosing level $n$-structure, $A$ defines a point on $\mathcal{A}_{g,n}.$ Now the Hodge bundle $\omega$ on $\mathcal{A}_{g,n}$ extends to an ample line bundle on the minimal compactification of $\mathcal{A}_{g,n}$. This gives us an adelic line bundle $\ovl{\omega} \in \wh{\Pic}(\mathcal{A}_{g,n}/k_0)$. We define the Faltings height $h_{\omega}$ for non-archimedean GVFs to be the height induced by this ample adelic line bundle $\ovl{\omega}$. Since $\omega$ is ample, we have 
$$h_{\Fal}(A) = 0 \implies A \text{ is defined over } k$$
where $k$ is the height zero subfield of $K.$ 

We first begin the proof in the case where $A$ has bad reduction. In this case, we can apply the equidistribution arguments of Ullmo \cite{Ulm98} and Zhang \cite{Zha98}, which have been adapted to function fields by Gubler \cite{gubler2007bogomolov} and Yamaki \cite{yamaki2013geometric}. We will construct a height function $h_{\partial \mathcal{A}_{g,n}}$ on $\mathcal{A}_{g,n}.$ that measures bad reduction of $A$. To do this, fix a toroidal compactification $\ol{\mathcal{A}_{g, n}}$ of $\mathcal{A}_{g, n}$ over $\Spec \mathbb{Z}[1/n, \zeta_n],$ and let $\partial \mathcal{A}_{g, n}$ be $\ol{\mathcal{A}_{g, n}} - \mathcal{A}_{g, n}.$ Note that this is a relative effective Cartier divisor (see \cite[Theorem 6.7]{chaifaltings}), so we may then consider the line bundle $\mathcal{O}(\partial \mathcal{A}_{g, n})$ on $\ol{\mathcal{A}_{g, n}}.$ We can then restrict this data to lie over $k_0$. We hence obtain a model adelic line bundle $\ol{\mathcal{O}(\partial \mathcal{A}_{g, n})} \in \wh{\Pic}(\mathcal{A}_{g,n}/k_0)$. We may therefore define the height $h_{\partial \mathcal{A}_{g, n}}$ by taking the height functional induced by $\ol{\partial \mathcal{A}_{g,n}}$. For convenience, we abbreviate this as $h_{\partial \mathcal{A}_g}$.

\begin{proposition} \label{prop: boundary height and places of bad reduction}
View $K = (K, (\Omega, \mathcal{A},\nu),\phi)$ as an adelic curve and assume that $K$ is a non-archimedean GVF. The set $\{\omega \in \Omega \mid A \text{ has bad reduction at } \omega\}$ is a measurable subset of $\Omega.$ It has positive measure if and only if $h_{\partial \mathcal{A}_{g}}(A) > 0$. 
\end{proposition}
\begin{proof}
Since $K$ is a non-archimedean GVF, the set of places $\omega \in \Omega$ with residue characteristic different from the characteristic of $K$ has measure zero. We may thus remove those places and assume the residue characteristic of all $\omega \in \Omega$ is equal the characteristic of $K$.
\par 
Fix the data of $h_{\partial \mathcal{A}_{g, n}}$ as above, and consider the section $s$ of $\mathcal{O}(\partial \mathcal{A}_{g, n})$ on $\ol{\mathcal{A}_{g, n}}$ which cuts out $\partial \mathcal{A}_{g, n}.$ Note that for an absolute value $\omega$ for our adelic curve $K,$ $A$ has good reduction at $\omega$ if and only if, if $R$ is the valuation ring corresponding to the absolute value $\omega,$ the morphism $\Spec K \rightarrow \mathcal{A}_{g, n}$ extends to $\Spec R.$ In particular, as $n$ is not divisible by the characteristic, this occurs if and only if the pullback of $s$ to $\Spec K$ has norm at least $1$ with respect to $\phi(\omega)$ when we take the analytification of the adelic line bundle (see \cite[\S 3.3.1]{yuan2021adelic}).
\par Therefore, if we pullback the adelic line bundle $\ol{\partial \mathcal{A}_{g, n}}$ to $K$ and consider this section $s,$ we see that the set of places $\omega \in \Omega$ such that $A$ has bad reduction at $\omega$ where the residue characteristic of $\omega$ doesn't divide $n$ is precisely the set \[\{\omega \in \Omega| \,\, \|s\|_{\ol{\partial \mathcal{A}_{g, n}}, \omega} < 1\},\] which is measurable since Yuan--Zhang adelic line bundles induce Chen--Moriwaki adelic line bundles. 
\par 
Since 
$$h_{\partial \mathcal{A}_{g,n}} = \int_{\Omega} \log ||s||^{-1}_{\omega} d \nu(\omega),$$
and $\log ||s||^{-1}_{\omega} \geq 0$ as $s$ is an effective section of $\ol{\mathcal{A}_{g,n}}$, it follows that $h_{\partial \mathcal{A}_{g,n}}(A) > 0$ if and only if there is a set of positive measure in $\Omega$ for which $A$ has bad reduction.
\end{proof}

\begin{proposition} \label{GVFDivisor2}
In the setting of Theorem \ref{GVFDivisor1}, assume that $h_{\partial \mathcal{A}_g}(A) > 0$. Then Theorem \ref{GVFDivisor1} is true.    
\end{proposition}

Before we turn to the proof of Proposition \ref{GVFDivisor2}, we have the following useful lemma.
\begin{lemma} \label{lem: image of height zero}
Suppose that $\phi: A \rightarrow B$ is a homomorphism of abelian varieties. Then, the image of a subvariety with canonical height zero (with respect to any ample symmetric line bundle) has canonical height zero.
\end{lemma}
\begin{proof}
For the case of dimension zero subvarieties, this follows from the argument in Lemma \ref{lem: dimension zero bogomolov}. In general, say that $V$ is a positive-dimensional subvariety of $A.$ As $\h_{A,L}(V) = 0$, by \cite[Proposition 2.6.8]{Hul26}, after passing to a countable GVF extension, we obtain a Zariski dense set of height zero points for $V$. Then $\phi(V)$ also contains a Zariski dense set of height zero points, and hence the canonical height of $q(V)$ must be zero by \cite[Proposition 2.6.10]{Hul26}. 
\end{proof}

\begin{proof}[Proof of Proposition \ref{GVFDivisor2}]
First, we may assume that $\Stab(D)$ is finite. Otherwise, if $B$ is the stabilizer, we may work with the image of $D$ in $A/B$ instead, where we now have finite stabilizer. Now as $\h_L(D) = 0$ and $L$ is ample, by Proposition \ref{SmallPointsExistence1}, we may pass to a countable extension $K'/K$ such that $D(K')$ contains a generic sequence $(x_n)$ of small points. Then, base changing to the GVF $K',$ $h_{\partial \mathcal{A}_g}(A)$ is still positive, and hence we still have a positive measure set of bad reduction places.  
\par We now apply Theorem \ref{YamakiGVF1}. As $\Stab(D)$ is finite and $A_v$ has bad reduction, we certainly always have $ \dim(D/\Stab(D)) = g-1 \geq b(A_v)$. Hence $D$ cannot have a dense set of small points unless it is a special subvariety. Since $A$ is traceless, this means that $D$ is a torsion translate of an abelian subvariety, as desired.
\end{proof}

We are now left with the case where $h_{\partial \mathcal{A}_g}(A) = 0$, i.e. the case of almost everywhere good reduction. It is useful to analyze the case of global function fields first, say $K = \bb{C}(B)$ where $B$ is a curve over $\bb{C}.$ In this case, we have good reduction literally everywhere. Furthermore, let's assume $D$ is symmetric, ample, and endow $L = O(D)$ with the canonical metric after fixing a rigidification $0^*L \simeq K$, where $0: K \to A$ is the zero section. Let $s \in H^0(A,L)$ be the global section corresponding to $D$. We have the formula 
$$\h_{A,L}(D) = \sum_{v \in M_K} \int \log |s(x)|_v d \mu_{A,v}.$$
Since $A$ has good reduction at $v$ for all places $v,$ $\mu_{A,v}$ is supported exactly at the Shilov point and hence $\log |s(x)|_v = \log \sup\limits_{x \in A^{\an}_v} |s(x)|_v$. Since $\h_{A,L}(D) = 0$, we have 
$$\sum_{v \in M_K} \log \sup\limits_{x \in A^{\an}_v} |s(x)|_v = 0.$$
This consequence has already been observed by Yamaki \cite{Yam17} in his proof of Bogomolov conjecture for divisors in the everywhere good reduction case.
\par 
On the other hand, the key formula of Moret-Bailly \cite{MB85}, in the reformulation of Bost \cite{Bos96} (see \cite[Equation A.4]{GG26} for the function field case), gives us 
\begin{equation} \label{eq: MoretBailly1}
\wh{\mu}(\pi_* \ovl{L}) = -\frac{1}{2} h_{\Fal}(A).
\end{equation}
Here, $\wh{\mu}$ is the average slope of the adelic vector bundle $\pi_* \ovl{L}$, where an element $x \in \pi_* \ovl{L}$ corresponding to a global section $s$ is given the metric $|x|_v := \sup |s|_v$. If $\dim H^0(A,L) = 1$, say if $D$ is the theta divisor of a principally polarized abelian variety, then we automatically get that 
$$-\frac{1}{2} h_{\Fal}(A) = 0$$
which implies that $A$ is isotrivial. For a general symmetric divisor $D$, if $K$ is of characteristic zero then we can get a contradiction by using semistability of the metrized vector bundle $\pi_* \ovl{L}$ \cite{Bos96b} which tells us that $\hat{\mu}(\pi_*\ol{L})$ bounds above the slope of any line (in particular, the line generated by $s,$ whose slope is zero). If $K$ is of positive characteristic, then the situation is more complicated but one can still arrive at a contradiction by constructing many small sections of $\ovl{L}^2$. 
\par 
For a general countable GVF $K$, viewing $K$ as an adelic curve instead, we may view $\pi_* \ovl{L}$ as an adelic vector bundle in the sense of Chen--Moriwaki \cite[Theorem 6.1.13]{CM20}. However, the formula (\ref{eq: MoretBailly1}) is not known to hold in the setting of a general adelic curve and indeed, it seems that would require understanding various compactifications of symmetric line bundles $\mathcal{L}$ on an abelian scheme $\mathcal{A}/S$ when $S$ is only quasi-projective (see \cite{BGHJ26} for the case of the principal theta divisor).
\par 
However, luckily for us, we only have to handle the case where $A/K$ has almost everywhere good reduction. For an abelian scheme $\mathcal{A}/S$, Tang \cite{Tan14} has proven that given a symmetric rigidified relatively ample line bundle $L$ on $\mathcal{A}$, we have a canonical isomorphism 
$$\det(\pi_*L)^{\otimes 24} \simeq (e^* \omega_A^{\vee})^{\otimes 12 d}$$
where $d$ is the rank of the locally free sheaf $\pi_* L$ and $e: S \to A$ is the zero section. Using this, we prove the following. 

\begin{proposition} \label{GoodReductionKeyFormula1}
Let $(K, (\Omega, \mathcal{A},\nu),\phi)$ be an adelic curve and let $A/K$ be an abelian variety. Let $L$ be a rigidified symmetric ample line bundle on $A$ and endow $L$ with the canonical metric. Then there is an isomorphism 
$$\det(\pi_* L)^{\otimes 24} \simeq (e^* \omega_A^{\vee})^{\otimes 12 d}$$
such that for every non-archimedean place $v \in \Omega$ of good reduction, the isomorphism extends to an isometry 
$$\det(\pi_* \ovl{L}_v)^{\otimes 24} \simeq (e^* \ovl{\omega}_{A,v}^{\vee})^{\otimes 12 d}.$$
\end{proposition}

\begin{proof}
This follows immediately from \cite[Theorem 1.3]{Tan14} and the universal nature of the construction given in \cite[\S 3]{Tan14}. Indeed choosing his canonical isomorphism over $K$, if $v$ is a place of good reduction for $A$, we let $\mathcal{A}_v, \mathcal{L}_v, \mathcal{\omega}_{\mathcal{A}_v}$ be the canonical models for $(A,L,\omega_A)$. Then, the canonical isomorphism given by \cite[Theorem 1.3]{Tan14} between $\det(\pi_* \mathcal{L}_v)^{\otimes 24}$ and $(e^* \mathcal{\omega}_{\mathcal{A}_v}^{\vee})^{\otimes 12 d}$ must induce our canonical isomorphism over $K,$ giving us the isometry we desired.
\end{proof}

As a consequence, we obtain the following. 

\begin{proposition} \label{SmallBasis1}
Let $(K, (\Omega, \mathcal{A}, \nu),\phi)$ be a non-archimedean adelic curve and let $A/K$ be an abelian variety. Assume that $A$ has almost everywhere good reduction and let $L$ be a symmetric ample line bundle and endow $L$ with its canonical adelic metric $\ovl{L}$. Then if $H^0(A,L)$ has a basis $\{s_1,\ldots,s_N\}$ with 
\begin{equation} \label{eq: SmallAverage1}
\int_{\Omega} \log \sup |s_i|_v d \nu(v) \leq 0
\end{equation}
for all $i$, we must have $h_{\Fal}(A) = 0$. 
\end{proposition}

\begin{proof}
By Proposition \ref{GoodReductionKeyFormula1}, viewing $\det(\pi_* \ovl{L})$ as a Chen-Moriwaki adelic line bundle \cite[\S 4.1.1.7]{CM20}, we have
$$24 \wh{\deg} \det(\pi_*L) = (-12 d) \wh{\deg}(e^*\ovl{\omega_A}),$$ since we have an isometry of adelic line bundles over almost every place.
As $\wh{\deg}(e^*\ovl{ \omega_A}) = h_{\Fal}(A)$ (see for example \cite[Section 5.5]{yuan2021adelic}), we have
$$\wh{\deg} \det(\pi_*L) = -\frac{d}{2} h_{\Fal}(A).$$
Now using our basis of sections $s_1,\ldots,s_N$, we obtain a section $s_1\wedge \cdots \wedge s_N$ of $\det(\pi_*L)$ which satisfies 
$$|s_1 \wedge \cdots \wedge s_N|_v \leq \prod_{i=1}^{N} \sup |s_i|_v$$
for all places $v$. Hence 
$$\int_{\Omega} \log |s_1 \wedge \cdots \wedge s_N|_v^{-1} d \nu(v) \geq -\int_{\Omega} \left(\sum_{i=1}^{N} \log \sup |s_i|_v \right) d \nu(v) \geq 0.$$
We therefore conclude that 
$$-\frac{d}{2} h_{\Fal}(A) \geq 0 \implies h_{\Fal}(A) = 0$$
as desired (since the Faltings height is nonnegative). 
\end{proof}

We are now ready to begin the proof of the Bogomolov conjecture over GVF for divisors in the (almost) everywhere good reduction case. We will say a section $s$ has small average if it satisfies (\ref{eq: SmallAverage1}). 

\begin{theorem} \label{GVFDivisor3}
Assume that $A/K$ has almost everywhere good reduction. Then Theorem \ref{GVFDivisor1} is true.
\end{theorem}

\begin{proof}
We may assume that $K$ is algebraically closed. As in the proof of Theorem \ref{GVFDivisor2}, we may assume that $\Stab(D)$ is finite. Then $D$ is an ample divisor. Now by replacing $D$ with $D + [-1]^*D$, we may assume that $D$ is symmetric, at the cost of losing the integrality of $D$. Let $L = O(D)$ and endow $L$ with the canonical metric by fixing a rigidification $e^*L \simeq K$ where $e: K \to A$ is the zero section. Let $s$ be the canonical (up to scaling) section of $L$ with $\divv(s) = D$. View $K = (K,(\Omega,\mathcal{A},\nu),\phi)$ as an adelic curve. Therefore, we have
$$\h_{A,L}(D) = \h_{A,L}(A) + \int_{\Omega} \left(\int_{A^{\an}_v} \log |s(x)|_v d \mu_{A,v} \right) d \nu(v) = \int_{\Omega} \log \sup_{x \in A^{\an}_v} |s(x)|_v d\nu(v),$$
where the second equality follows because $A$ has good reduction at every place outside a set of measure zero. Now by assumption, $\h_{A,L}(D) = 0$ and since $\h_{A,L}(A)  = 0$, we obtain 
$$\int_{\Omega} \log \sup |s(x)|_v d \nu(v) = 0$$
and hence $s$ has small average. We now aim to construct many small sections of the adelic line bundle $\ovl{L}^2$. This step appears in Yamaki's proof (\cite{Yam17}). To do this, fix a prime $\ell$ which is different from the characteristic $p$ of $K.$ Then, for $q \in A[\ell^{\infty}],$ let $t_q: A_K \rightarrow A_K$ be the translation by $q$ isomorphism, and let $D_q = t_q^*D + t_{-q}^*D.$ These are all divisors which are symmetric and linearly equivalent to $L^{\otimes 2},$ whose section is given by $s(x + q)s(x-q);$ we abbreviate this as $s_q.$ Then as $\sup |s(x+q)|, \sup |s(x-q)| = \sup |s(x)|$, the sections $s(x+q),s(x-q)$ certainly have small average and so does $s_q$. 

\begin{proposition} \label{DivisorSupport1}
There exists $q_0,\ldots,q_g \in A[\ell^{\infty}]$ such that $D_{q_0} \cap \cdots \cap D_{q_g} = \emptyset$. 
\end{proposition}

\begin{proof}
Observe that given $x \in A(K)$, the set of $q \in A(K)$ such that $x \in D_q$ forms a proper subvariety. We now construct inductively $q_0,\ldots,q_g$ such that for $0 \leq j \leq g,$ the irreducible components $\cap_{i=0}^{j} D_{q_i}$ all have codimension $j.$ Given $D_{q_0},\ldots,D_{q_j}$, we choose points $x_1,\ldots,x_N$ from each of the irreducible components of $\cap_{i=0}^{j} D_{q_i}$. Since $A[\ell^{\infty}]$ is Zariski dense in $A$, we may find $q_{j+1} \in A[\ell^{\infty}]$ such that $D_{q_{j+1}}$ does not contain any of the $x_i$'s. Then $D_{q_{g+1}}$ intersects any irreducible component of $\cap_{i=0}^{j} D_{q_i}$ properly, and hence the irreducible components of $\cap_{i=0}^{j+1} D_{q_i}$ has codimension $j+1$ as desired, finishing the induction. 
\end{proof}

Proposition \ref{DivisorSupport1} gives us sections $s_0,\ldots,s_g$ of $\ovl{L}^2$ which all have small average, such that they form a basepoint free set, and thus we obtain a morphism $f: A \rightarrow \mathbb{P}^g_K$ given by the sections $s_0, s_1, \ldots, s_g.$ We now want to produce enough sections to obtain a closed embedding, not just basepoint freeness. 

\begin{proposition} \label{DivisorSupport2}
There exists finitely many torsion points $x_1,\ldots,x_m \in  A[\ell^{\infty}](K)$ such that $\prod_{i=1}^{m} f_{x_i} : A \to (\bb{P}^g_{K})^m$ factors into $A \to B \xhookrightarrow{} (\bb{P}^g_{K})^m$, where $B$ is an abelian variety over $K$, $\pi: A \to B$ is an isogeny and $f_x$ is given by $f \circ t_x: A \rightarrow \bb{P}^g.$
\end{proposition}

\begin{proof}
We have to be careful in this proposition as $K$ may be of positive characteristic and hence an injection on the level of points need not imply a closed embedding due to inseparability. 
\par 
For each $x \in A(K),$ let $\Gamma_x = (f_x,f_x)^{-1}(\Delta) \subseteq A \times A,$ where $\Delta$ is the diagonal of $\bb{P}^g \times \bb{P}^g$. Then $\Gamma_x$ is a closed subscheme of $A^2$. We will consider the intersection $\bigcap_{x \in A[\ell^{\infty}](K)} \Gamma_x$ which is a closed subscheme. We define 
$$H = \{ a \in A \mid f \circ T_a = f \}.$$
This is a closed subscheme of $A$, where given an affine scheme $S$ we have 
$$H(T) = \{ a \in A(S) \mid f_S \circ T_a = f_S\}.$$
Note that $H$ must factor through $T_x^{-1}(f^{-1}(f(x)))$ for all $x \in A(K),$ and hence is a finite group scheme, since the fiber of $f$ over some point is finite (for instance by using the construction in Proposition \ref{DivisorSupport2}). We define $B = A/H$ so that we have an isogeny $\pi: A \to B$. 

\begin{lemma} \label{DivisorSupport3}
View $A \times_B A$ as a closed subscheme of $A \times A$. We have $R = A \times_B A$. 
\end{lemma}

\begin{proof}
Let $S$ be an affine $K$-scheme. Recall that $R = \cap_{x \in A[\ell^{\infty}](K)} \Gamma_x$, where 
$$\Gamma_x(S) = \{(a,b) \in A(S)^2 \mid f_S(a+x_S) = f_S(b+x_S)\}.$$
By definition it is clear that $(A \times_B A)(S) \subseteq \Gamma_x(S)$ for any $x$ and so we have $A \times_B A \subseteq R$. 
\par 
For the reverse inclusion, let $(a,b) \in R(S)$. Set $h = b-a \in A(S)$. Then by definition, we have 
$$(f_S \circ T_h)(x_S+a) = f_S(x_S+a) \text{ for all } x \in A[\ell^{\infty}](K)$$
where $x_S$ denotes the base change to $S$. We claim that $\{x_S + a \}$ as $x$ ranges over $A[\ell^{\infty}](K)$ forms a scheme-theoretic dense set. If this were true, then $f_S \circ h = f_S$ and hence we get the reverse inclusion $R \subseteq A \times_B A$, giving us equality $R = A \times_B A$ as desired. 
\par It now suffices to prove that $\{x_T\}$ is scheme-theoretically dense in $A_T$. But $A[\ell^{\infty}](K)$ is Zariski dense in $A(K)$ and hence is scheme-theoretically dense in $A$ as $A$ is reduced. Then scheme-theoretic density in $A_T$ follows from \cite[Lemma IX.4.4]{SGA3}. 
\end{proof}

Using Lemma \ref{DivisorSupport3}, by Noetherianness we may find finitely many $x_1,\ldots,x_m \in A[\ell^{\infty}](K)$ such that $A \times_B A = \cap_{i=1}^{m} \Gamma_{x_i}$. If we let $f_i$ denote the morphism $f_{x_i}$, then $\prod_{i=1}^{m} f_i: A \to (\bb{P}^g)^{m}$ factors through $A \to B \to (\bb{P}^g)^m$. We claim that $B \to (\bb{P}^g)^{m}$ has to be a closed immersion. By \cite[\href{https://stacks.math.columbia.edu/tag/04XV}{Tag 04XV}]{stacks-project}, as $B$ is projective, it suffices to show that it is a monomorphism. This can be checked by using test schemes $T$, in which case it is straightforward using the fact that $A \times_B A = \cap_{i=1}^{m} \Gamma_{x_i}$, finishing the proof of our Proposition.
\end{proof}

Using the Segre embedding, we may embed $(\bb{P}^g)^m$ into $\bb{P}^{(g+1)^m-1}$. Let $M$ be the pullback of $O(1)$ to $B$, so that $\pi^* M =\otimes_{i=1}^{m} t_x^* L^2$. Let $M' = M \otimes [-1]^*M$, so that $M'$ is symmetric. Then 
$$\pi^* M' \simeq \bigotimes_{i=1}^{m} \left(t_x^* L^2 \otimes t_{-x}^* L^2 \right) \simeq \bigotimes_{i=1}^{m} L^4 \simeq L^{4m},$$ 
where the second isomorphism is by the theorem of the square. Endowing $M'$ with a canonical adelic metric $\ovl{M}$, we obtain that $\pi^* \ovl{M}' = \ovl{L}^{4m}$. 
\par 
Now using the embedding $B \xhookrightarrow{} (\bb{P}^{(g+1)^m-1})$, we obtain a set of sections $\{s_1,\ldots,s_N\}$ for $M$, such that under pullback to $A$ by $\prod_{i=1}^{m} f_i$, they have small average for the adelic line bundle $\otimes_{i=1}^{m} t_x^* \ovl{L}^2$. Applying $[-1]^*$ gives us  a corresponding set of sections $\{s'_1,\ldots,s'_N\}$ for ${-1}^*M$ whose pullback to $A$ have small average for $\otimes_{i=1}^{m} t^*_{-x} \ovl{L}^2$. 
\par Now observe that products of small average sections still have small average. Hence $\{s_i \otimes s'_j\}$ give us a set of small average sections for $\ovl{M}'$. Furthermore, they induce a closed immersion of $B$ to projective space, say $\bb{P}^{N}$. Then by ampleness, we know that for some large enough $n$, we have a surjection $H^0(\bb{P}^N,O(n)) \to H^0(B,(M')^n)$. On the other hand, taking products of $\{s_i \otimes s'_j\}$ gives us a basis of $H^0(\bb{P}^N,O(n))$ with small average, and hence $(\ovl{M}')^n$ is a symmetric adelic line bundle on $B$ with a basis of small average sections. Applying Proposition \ref{SmallBasis1} gives us that $h_{\Fal}(B) = 0$. This implies that $B$ is defined over the height zero subfield, contradicting the fact that $A$ is traceless as desired.  
\end{proof}

\section{Height Equidistribution and Inequalities} \label{sec: HeightEquidistribution}
The aim of this section is to extend two known Arakelov theoretic results from the case of number/function fields to the setting of general countable GVFs. The first is the following height equidistribution property over GVFs. Let $K$ be a non-archimedean GVF and let $X$ be a projective variety over $K$. Let $\ovl{L},\ovl{M}$ be two adelic line bundles on $X$ so that $L$ and $M$ are both big. Assume there exists some $a > 0$ so that $a \ovl{L} - \ovl{M}$ is nef. Let $D$ be a geometrically integral divisor on $X,$ and assume that $h_{\ovl{L}}(X) = 0$. Then we wish to prove that
\begin{equation} \label{eq: HeightEqui1}
h_{\ovl{L}}(D) = h_{\ovl{M}}(D).
\end{equation}
We view this as a more refined version of equidistribution which gives us information even if the related Chambert-Loir measures for $\ovl{L}$ and $\ovl{M}$ are trivial. When $K$ is a global field, this would follow by using \cite[Lemma 6.1]{CLT09} and the fact that there is a generic sequence of points $(x_n)$ that is simultaneously small for both $\ovl{L}$ and $\ovl{M}$. S{\'e}dillot \cite[Proposition 5.5]{Sed23} gives an extension to this for adelic curves but he requires the assumption that $L$ is ample and $D$ is a Cartier divisor. The ampleness condition is fine for our application but we would like to remove the Cartier assumption. The usual approach by Chambert-Loir and Thuillier uses the arithmetic Siu inequality by Yuan \cite{Yua08} which seems to be not known in the general adelic curve case yet. 
\par 
The second result that we need is an arithmetic Brunn--Minkowski type inequality. Given nef semipositive line bundles $\ovl{L},\ovl{M}$ on $X$ such that $L,M$ are both big, we have
\begin{equation} \label{eq: BrunnMinkowski1}
\ovl{L}^d \cdot \ovl{M} = 0 \implies \ovl{M}^{d+1} = 0.
\end{equation}
Over number fields, this is implied by a result of Chen \cite[Theorem 1.1]{Che18} which is a refinement of the arithmetic isoperimetric inequality and the arithmetic Brunn--Minkowski inequality. The arithmetic Brunn--Minkowski inequality was established previously by Yuan \cite{Yua09} and as noted, can be viewed as a higher dimensional generalization of the arithmetic Hodge index theorem. 
\par 
As we were unable to find a reference for a function field version of Chen's result, although the same proof most likely works, we instead rederive both (\ref{eq: HeightEqui1}) and (\ref{eq: BrunnMinkowski1}) from the arithmetic Hodge index theorem for GVFs. Note that we only need the inequality part of the arithmetic Hodge index, which is much easier than determining the equality cases. 
\par 
The arithmetic Hodge index theorem states that if $\ovl{L}_1,\ldots,\ovl{L}_{n-1}$ are nef adelic line bundles with each $L_i$ big and $\ovl{M}$ an integrable adelic line bundle, then 
$$M \cdot L_1 \cdots L_{n-1} = 0 \implies \ovl{M}^2 \cdot \ovl{L}_1 \cdots \ovl{L}_{n-1} \leq 0.$$
For number fields this is by Yuan \cite{YZ17} and for function fields, this is by Carney \cite{Car20}. For GVFs, we prove this using a continuity argument on the GVF analytification. Recall that $k_0$ is $\bb{Q}$ if $K$ is characteristic zero and $\bb{F}_p$ otherwise. First, we may find a quasi-projective variety $U$ over $k_0$, a flat family $\mathcal{X} \to U$ and adelic line bundles $\ovl{\mathcal{L}}, \ovl{\mathcal{M}} \in \wh{\Pic}_{\operatorname{int}}(\mathcal{X}/k_0)$ such that $\Spec K$ maps to the generic point of $U,$ and we recover back $(X,L,M)$ after base changing to $\Spec K$. Note that if $\ovl{L}$ is nef, then we can find such a model with $\ovl{\mathcal{L}}$ being nef and hence $\iota^* \ovl{\mathcal{L}}$ is nef for any map $\iota: \Spec L \to U$.

\begin{proposition} \label{HodgeIndex1}
Assume that $\ovl{L}_1,\ldots,\ovl{L}_{n-1}$ are nef and geometrically big line bundles. Then for any integrable $\ovl{M}$, we have
$$M \cdot L_1 \cdots L_{n-1} = 0 \implies \ovl{M}^2 \cdot \ovl{L}_1 \cdots \ovl{L}_{n-1} \leq 0.$$
\end{proposition}

\begin{proof}
Let $\mathcal{X} \to \mathcal{U}$ be our flat family. Since big line bundles are exactly those which can be written as $A \otimes E$, where $A$ is ample and $H^0(X,E) > 0$ \cite[Corollary 2.2.7]{Laz04}, by further shrinking $U$, we may assume that for every $s \in \mathcal{U}$, the line bundle $\mathcal{L}_{i,s}$ is big on the fiber $\mathcal{X}_s$. By flatness, we may also assume that $\mathcal{M}_s \cdot \mathcal{L}_{1,s} \cdots \mathcal{L}_{n-1,s}= 0.$ We form the Deligne pairing 
$$\ovl{\mathcal{N}} = \langle \mathcal{M}, \mathcal{M}, \mathcal{L}_1,\ldots,\mathcal{L}_{n-1} \rangle \in \wh{\Pic}(\mathcal{U}/k_0).$$
We now pick some $t \in \Frac(U) \setminus k_0$ for which $h(t) > 0$, which we may assume exists as otherwise the GVF structure is trivial and every intersection product is zero. Further shrinking $\mathcal{U}$, we may assume that $t$ is regular so that we obtain a map $\pi: \mathcal{U} \to \bb{A}^1$ induced by $t$. We let $\mathcal{U}_0$ denote its generic fiber so that it is a variety over $k(t)$. 
\par 
We now consider the GVF analytification $\mathcal{U}^{\GVF}_0$ over the base GVF $k(t).$ Taking the height of the adelic line bundle $\ovl{N}$ defines a continuous function for us 
$$h_{\ovl{N}}: \mathcal{U}^{\GVF}_0 \to \bb{R}.$$
Furthermore, we know from the function field arithmetic Hodge index that $h_{\ovl{N}}(x) \leq 0$ when evaluated at any $\ovl{k(t)}$-point $x$. By existential closedness (\cite[Theorem 2.1]{BH22}), such points are dense and so $h_{\ovl{N}} \leq 0$ by continuity, proving our arithmetic Hodge index for the GVF point $\Spec K$ inside $\mathcal{U}_0^{\GVF}$. 
\end{proof}

Using Proposition \ref{HodgeIndex1}, we can prove the following mixed-intersections inequality which will be useful for us later. 

\begin{proposition} \label{HodgeIndex2}
Let $X$ be a projective variety over $K$ and let $\ovl{L}$ and $\ovl{M}$ be nef semipositive adelic line bundles which are geometrically big. Then if $d = \dim X$, we have
$$\ovl{L}^d \cdot \ovl{M} = \ovl{L}^{d+1} =  0 \implies \ovl{M}^{d+1} = 0.$$
\end{proposition}

\begin{proof}
By replacing $\ovl{M}$ with $\ovl{M}_{\eps} = \ovl{M} + \eps \ovl{L}$, we may assume that the geometric degrees $d_i = L^{n-i} \cdot M^i$ are all positive. We define $I_j = \ovl{L}^{n+1-j} \cdot \ovl{M}^{j}$, so that $I_0 = I_1 = 0$ and we wish to prove that $I_{n+1} = 0$. In fact we will prove that $I_j = 0$ for all $0 \leq j \leq n+1$. 
\par 
To do so, we set
$$c_j = \frac{d_{j+1}}{d_j}, \quad N_j =  M - c_j L,\quad \ovl{N}_j = \ovl{M} - c_j \ovl{L}.$$
Then 
$$N_j \cdot L^{n-1-j} \cdot M^j = 0$$
so that we can apply Proposition \ref{HodgeIndex1} to $\ovl{N}_j$ with $(n-1-j)$ copies of $\ovl{L}$ and $j$ copies of $\ovl{M}$. We obtain 
$$(\ovl{M} - c_j \ovl{L})^2 \ovl{L}^{n-1-j} \ovl{M}^{j} \leq 0.$$
Expanding out, we obtain $I_{j+2} - 2 c_j I_{j+1} + c_j^2 I_j \leq 0$. Now by induction, starting with $I_1 = I_0 = 0$, we have 
$$I_j, I_{j+1} = 0 \implies I_{j+2} \leq 0 \implies I_{j+2} = 0$$ 
as by nefness, we must have $I_{j+2} \geq 0$. Hence $I_{n+1} = 0$, finishing our proof. 
\end{proof}

We now prove our height equidistribution property for divisors. If $X$ is a projective variety, then for any effective divisor $D$ we prove the following.

\begin{proposition} \label{HeightEquidistribution2}
Let $X$ be a projective variety over $K$ and $\ovl{L}, \ovl{M}$ two nef line bundles such that $L$ and $M$ are both geometrically semiample and big. Assume there exists some $a > 0$ such that $a \ovl{L} - \ovl{M}$ is nef. Then $h_{\ovl{L}}(X) = 0$ implies that for any geometrically integral effective divisor $D$ of $X$, we have 
$$\frac{ h_{\ovl{L}}(D)}{\deg_L(X)} = \frac{h_{\ovl{M}}(D)}{\deg_M(X)}.$$
where $h_{\ovl{L}}, h_{\ovl{M}}$ are the un-normalized heights (that is, $h_{\ovl{L}}(D) = \widehat{\deg}(\ol{L}^d|_D)$).  
\end{proposition}

\begin{proof}
Let $d = \dim X$. We first check that all mixed intersections between $\ovl{L}$ and $\ovl{M}$ vanish on $X.$ Indeed since $a \ovl{L} - \ovl{M}$ is nef, it follows from $\ovl{L}^{d+1} = 0$ that $(\ovl{L} + \ovl{M})^{d +1} = 0$ too. Expanding out, as each term is non-negative by nefness, it follows that each mixed intersection must be zero. 
\par 
We now prove that for any integrable adelic line bundle $\ovl{N}$, we have 
\begin{equation} \label{eq: DivisorDegree1}
\frac{\ovl{L}^d \cdot \ovl{N}}{\deg_L(X)} = \frac{\ovl{M}^d \cdot \ovl{N}}{\deg_M(X)}.
\end{equation}
Let $\ovl{L}_t = (1-t) \ovl{L} + t \ovl{M}$ and let 
$$F(t) = \frac{\ovl{L}_t^d \cdot \ovl{N}}{L_t^d}.$$ 
Let $\ovl{\Delta} = \ovl{M} - \ovl{L}$ so that $\Delta = M - L$. We set 
$$c_t = \frac{L_t^{d-1} \cdot \Delta}{L_t^d}$$
so that if $\ovl{D}_t = \ovl{\Delta} - c_t \ovl{L}_t$, by definition we have $D_t \cdot L_t^{d-1} = 0$. Now given our $\ovl{N}$, we may assume that $L_t^{d-1} \cdot N = 0$ by adding a multiple of $\ovl{L}$ to $\ovl{N}$, which won't change the quantity $\ovl{L}^d \cdot \ovl{N}$ and $\ovl{M}^d \cdot \ovl{N}$.
\par 
Consider the quadratic form induced on the space of integrable adelic line bundles $\ovl{N}$, which satisfies $N \cdot L_t^{d-1}= 0$, by defining
$$\langle \ovl{N}_1, \ovl{N}_2 \rangle_t := \ovl{L}_t^{d-1} \cdot \ovl{N}_1 \cdot \ovl{N}_2.$$
Since $L_t$ is big and $\ovl{L}_t$ is nef and semi-positive, by Proposition \ref{HodgeIndex1} it follows that $\langle \bullet , \bullet \rangle_t$ is negative semi-definite. Observe that $\langle \ovl{D}_t, \ovl{D}_t \rangle_t = 0,$ since $\ol{D}_t$ is a linear combination of $\ol{L}, \ol{M},$ and all of the arithmetic intersections of these adelic line bundles are zero. Therefore, by negative semi-definiteness, it follows that $\langle \ovl{D}_t, \ovl{N} \rangle_t = 0$. Expanding out, this gives us $$\ovl{L}_t^{d-1} \cdot \ovl{\Delta} \cdot \ovl{N} = c_t \ovl{L}_t^d \cdot \ovl{N}.$$
We now set $A(t) = \ovl{L}_t^{d} \cdot \ovl{N}$ and $v(t) = L_t^{d}$. We may compute that 
$$A'(t) = d \ovl{L}_t^{d-1} \ovl{\Delta} \cdot \ovl{N} = dc_t \ovl{L}_t^d \cdot \ovl{N} = d c_t A(t).$$
Similarly we have $v'(t) = d c_t v(t)$. Hence $A'(t)v(t) = A(t)v'(t)$ for all $t,$ and so 
$$\frac{d}{dt} \left(\frac{A(t)}{v(t)} \right) = 0.$$
Evaluating at $t =0$ and $1$ gives us (\ref{eq: DivisorDegree1}) as desired.
\par 
We now replace $X$ with a blowup, so that $D$ is a Cartier divisor. Then $L$ and $M$ are still geometrically semiample and big. Let $N = O(D)$ and give $\ovl{N}$ an adelic metric so that it is integrable. We thus have 
$$
\frac{\ovl{L}^d \cdot \ovl{N}}{\deg_L(X)} = \frac{\ovl{M}^d \cdot \ovl{N}}{\deg_M(X)}.$$
On the other hand, if $s$ is the canonical section of $O(D)$, we have the formula 
\begin{equation} \label{eq: DivisorHeight1}
h_L(D) = \wh{\deg}(\ovl{L}^d \vert_D) = \ovl{L}^d \cdot \ovl{N} + \int_{\Omega_K} \left(\int_{X_v^{\an}} \log |s|_v d c_1(\ovl{L})^d_v \right)d \nu(v).
\end{equation}
and similarly for $M$. Here, we let $c_1(\ovl{L})^d_v$ be the un-normalized measure, i.e. it has total mass $\deg_L(X)$. We now apply Proposition \ref{EqualityMeasures1}, which holds for geometrically semiample and big line bundles. We hence obtain that outside a set of measure zero, we have $\frac{c_1(\ovl{L})_v^d}{\deg_L(X)} = \frac{c_1(\ovl{M})_v^d}{\deg_{M}(X)}$. Combining with (\ref{eq: DivisorHeight1}) gives us $\frac{h_L(D)}{\deg_L(X)} = \frac{h_M(D)}{\deg_M(X)}$ as desired.
\end{proof}

\section{Bogomolov Conjecture for Subvarieties over GVFs}

In this section, we will turn to proving the Bogomolov conjecture over arbitrary GVFs. As the archimedean case was already proven by Chen--Moriwaki in \cite{chen2022equidistribution}, we only need to handle the non-archimedean case.

\begin{theorem} \label{GVFBogomolov1}
Let $K$ be a non-trivial non-archimedean GVF and let $k_0$ denote its height zero subfield. Let $A$ be an abelian variety over $K$ and assume that $\Tr_{\ovl{K}/\ovl{k_0}}(A_{\ovl{K}}) = 0$. Let $V$ be a positive dimensional subvariety of $A$. Then 
$$\h_A(V) = 0 \implies V = x +B$$
where $x$ is a point of $A$ and $B$ an abelian subvariety. 
\end{theorem}

When $V$ is a divisor, this follows from the results of Section \ref{sec: DivisorGVF}. The aim is to reduce the general subvariety case to the case of divisors. For instance, suppose $A$ is simple and $V = C$ is a curve. If $\h_A(C) = 0,$ then the divisor $D = \underbrace{C + C + \cdots + C}_{g-1 \text{ times}}$ satisfies $\h_A(D) = 0,$ giving us a contradiction. In general, we can try to do an inductive argument similar to Xie--Yuan \cite{XY22}, whenever we have a $r > 0$ so that the addition morphism $p: V^r \to A$ has a generic fiber of positive dimension, but of dimension strictly less than $V$. 
\par The first case not covered by this type of argument is when $V$ is a surface inside an abelian fourfold $A.$ In this case one expects $p: V \times V \to A$ to be a generically finite morphism, and hence the generic fiber is dimension zero. We now explain how to construct our desired height zero divisor $D$ in this case. 
\par Now, assume first for simplicity that $p: V \times V \to A$ is a birational morphism. Then as it cannot be an isomorphism, a result of Kollar--Mori, \cite[Corollary 2.63]{KM08}, implies that there must exist an exceptional divisor $E$ on $V \times V$. Let $L$ be a symmetric ample line bundle on $A$ and let $\ovl{L}$ be the canonical adelic line bundle. Let $\ovl{P} = \pi_1^* \ovl{L} + \pi_2^*\ovl{L}$ and let $\ovl{N} = \delta^*\ovl{L}$. Using Proposition \ref{HeightEquidistribution2}, one obtains that
$$h_{\ovl{P}}(D) = h_{\ovl{N}}(D)$$
for any geometrically integral divisor $D$ on $V \times V$ as a consequence of $\h_A(V) = 0$. This step can be viewed as a refined version of equidistribution, which gives us information even when the measures are all trivial. Applying this to $E$, we see that since $h_{\ovl{N}}(E) = 0$ by the projection formula as $\delta$ contracts $E$. Hence $h_{\ovl{P}}(E) = 0$, giving us our height zero divisor. 
\par 
In general $p : V \times V \to A$ will not be birational. As it may however be finite, we switch to looking at $\delta: V \times V \to A$ given by $(x,y) \mapsto x-y$. This map is expected to be generically finite, but not finite as the diagonal $\Delta$ is sent to $0$ under $\delta$. We then use an algebraic geometry construction from \cite{ADK13}, which was suggested by GPT6 Astra. Let $X =  V \times V$ and view $K(X)$ as a finite extension of $K(A)$ using $\delta$. For simplicity, let's assume this extension is separable. We let $E/K(A)$ be the Galois closure of $K(X)/K(A)$ and we let $A'$ be the normalization of $A$ in $E$. Let $\sigma_i: K(X) \xhookrightarrow{} E$ for $i=1,\ldots,r$ be the different conjugate embeddings. Then we obtain rational maps $\tau_i : A' \dashrightarrow{} X$. Letting $X' = \underbrace{X \times_A \times \cdots \times_A X}_{r \text{ times }}$, we obtain a rational map $\tau: A' \dashrightarrow{} X'$ by combining the $\tau_i$'s. 
\par 
We now let $Y$ be the closure of the graph of $\tau$. This comes with projection maps $\pi_i: Y \to X$ and also maps $f: Y \to A$ given by $\delta \circ \pi_i$. Furthermore, the Galois group $G = \Gal(E/K(A))$ acts on $Y$. We now set $Z = Y/G$. Then $f: Y \to A$ descends to $\tilde{f}: Z \to A$ as a projective birational morphism. This map $\tilde{f}$ cannot be an isomorphism as otherwise $f: Y \to A$ is a finite morphism, which cannot be because $\delta$ contracts the diagonal. This produces for us our exceptional divisor. 
\par 
We now begin the proof of Theorem \ref{GVFBogomolov1} in the general case. We first proceed by induction on $(\dim V, \dim A)$, where if we are in the case of $(r,g)$, we assume that the cases $(a,b)$ has been proven for all $a < r$ and all $a = r, b < g$. Assume $V$ is a positive dimensional subvariety of $A$ with $\dim A = g$, $\deg_L(V) = d$ and $\h_{A,L}(V) = 0$. We may assume that $\Stab(V)$ is finite and that $V-V$ generates $A$, as otherwise we are done by induction (namely, by either quotienting $A$ or considering the abelian subvariety generated by $V - V$).  We may also assume that $K$ is countable, and is algebraically closed. Note that we may pass from $K$ to any other countable GVF extension $L$ which is algebraically closed, as the zero trace condition is maintained. Indeed if $k$ is the height zero subfield of $K$ and $k'$ that of $L$, then both $k'$ and $k$ are algebraically closed too. Applying Theorem 6.4(3) and Theorem 6.8 of \cite{Con06} gives us that $\Tr_{L/k'}(A_L)$ is still $0$ too. 
\par 
The case of $\dim V = 1$ can be deduced from the divisor case by the following argument. Let $B$ be the abelian subvariety that $C-C$ generates. After translating $C$ by a height zero point on $C$, which exists after passing to some countable GVF extension, we may assume that $C \subseteq B$. Then 
$$D = \underbrace{C + C + \cdots + C}_{\dim B - 1 \text{ times}}$$
is a divisor of $B$ with finite stabilizer and having canonical height $0$. Then it must be the translate of some abelian subvariety of $B$, contradicting the fact that $C$ generates $B$. 
\begin{proposition} \label{Induction1}
For any abelian subvariety $B \subseteq A$, let $q: A \to A/B$ be the quotient morphism. Then $q(V) = A/B$. 
\end{proposition}

\begin{proof}
Assume otherwise that $q(V)$ is a proper subvariety of $A/B;$ we know by Lemma \ref{lem: image of height zero} that $q(V)$ also has height zero. Now by induction, we know that $q(V) = x + B'$ for some abelian subvariety $B'$ of $A/B$. However $q(V)-q(V)$ must generate $A/B$ which is a contradiction unless $B' = A/B$. 
\end{proof}

Proposition \ref{Induction1} implies that $V$ is geometrically non-degenerate in the sense of Debarre--Moonen \cite[Definition 5.1]{DM25}. Then \cite[Corollary 5.6]{DM25} implies that
$$\dim(\underbrace{a_1V + \cdots + a_mV}_{m \text{ times }}) = \min(g , m (\dim V)),$$
where $a_1,\ldots,a_m$ are any positive integers. Let $m$ be the smallest positive integer so that $m (\dim V) \geq g$ and consider the addition map $f: V^m \to A$. 

\begin{proposition} \label{Induction2}
We have the equality $m(\dim V) = g$.  
\end{proposition}

\begin{proof}
Assume otherwise that $m(\dim V) > g$. Then the generic fiber of $f: V^m \to A$ has dimension $r$, where $0 < r< \dim V$. First, observe by Proposition \ref{Induction1} that $A$ cannot have an elliptic factor, as otherwise it is impossible for $q(V) = A/E$ since we may assume that $V$ is not a divisor by Theorem \ref{GVFDivisor1}. Thus all effective divisors of $A$ have positive canonical height. 
\par 
Now let $K' = K(A)$ be the generic fiber of $A$. We endow it with a polarized GVF structure \cite[Definition 2.3.5]{Hul26} coming from the adelic line bundle $\ovl{L}$, i.e. the GVF structure denoted by $K(A, \ovl{L},\ldots,\ovl{L})$. By definition, if $\ovl{\mathcal{D}}$ is an adelic divisor on a birational model of $A$, the GVF functional $\ell$ for the structure $K(A, \ol{L}, \ldots, \ol{L})$ is defined by 
$$\ell(\ovl{\mathcal{D}}) = \wh{\deg}(\ovl{L}^g \cdot \ovl{\mathcal{D}}).$$

\begin{lemma} \label{HeightZeroSubfield1}
The height zero subfield of $\ovl{K'}$ is still $k$.
\end{lemma}

\begin{proof}
This is \cite[Lemma 5.4.3]{Hul26}. His proof uses the Bogomolov conjecture for divisors of $A$ which we already have. In general if $X$ is a normal projective variety and $\ovl{L}$ an adelic line bundle on $X$, the same proof shows that the polarized GVF structure on $K(X)$ still has the same height zero subfield as long as all prime divisors of $X$ have positive height with respect to $\ovl{L}$.

\end{proof}
 
Let $W$ be the generic fiber for $f: V^m \to A$, so that $W$ can be viewed as a variety over $K'$. We may view it naturally as a subvariety of $A_{K'}^m$ using the embedding induced by $V^m \xhookrightarrow{} A^m$. We claim that $W$ has zero canonical height inside $A_{K'}^m$. Let $\ovl{M} = \pi_1^*\ovl{L} + \cdots + \pi_m^*\ovl{L},$ with underlying line bundle $M.$ Then by definition,
$$\h_{K'}(W) = \frac{\ell(\langle \ovl{M},\ldots,\ovl{M} \rangle_{V^m/A})}{M_{K'}^{\dim W} (\dim W + 1)},$$
where $\langle \ovl{M},\ldots,\ovl{M} \rangle$ gives us an adelic line bundle on $A$ and hence can be evaluated by our GVF functional $\ell$. We thus have
$$\h_{K'}(W) = \frac{\wh{\deg}_K(\langle \ovl{M},\ldots,\ovl{M} \rangle_{V^m/A} \cdot \ovl{L}^g)}{M_{K'}^{\dim W} (\dim W + 1)},$$
where $\langle \ovl{M},\ldots,\ovl{M} \rangle_{V^m/A} \cdot \ovl{L}^g$ is viewed as an adelic line bundle on $\Spec K$. 
We now view the morphisms $V^m \to A \to \Spec K$ as $X \to Y \to S$. We apply \cite[Lemma 4.6.1, (1)]{yuan2021adelic}, where after replacing $A$ with a Zariski open $U$, we obtain that $(V^m)_U \to U$ is flat. In this situation, we obtain an isomorphism 
$$\langle \ovl{M},\ldots,\ovl{M} \rangle_{(V^m)_U/U} \cdot \ovl{L}^g \simeq \langle \ovl{M},\ldots,\ovl{M}, f^*\ovl{L}^g \rangle_{(V^m)_U/\Spec K}.$$
Now as $f^*\ovl{L} \leq a \ovl{M}$ for some $ a > 0$ by the ampleness of $\ol{M},$ it follows that 
$$\langle \ovl{M},\ldots,\ovl{M}, f^*\ovl{L}^g \rangle_{(V^m)_U/\Spec K} \leq a' \langle \ovl{M},\ldots, \ovl{M} \rangle_{(V^m)_U/\Spec K}$$
for some $a' > 0$. But the latter is $0$ as $\h_{\ovl{M}}(V^m) = 0$. This proves that 
$$\h_{K'}(W) = 0$$
as desired. Now by induction, as the height zero subfield of $K'$ is still $k_0$ and $\Tr_{\ovl{K'}/\ovl{k_0}}(A_{K'}^m) = 0$, it follows that $W$ is a finite union of cosets of $A_{K'}^m$. 
In this case, since $W$ is irreducible (but not necessarily geometrically irreducible) as it is the fiber of the dominant morphism $f$ from an irreducible variety, the cosets forming $W$ must be translates of the same abelian subvariety $B$. This follows from the fact that every abelian subvariety of $A_{K'}$ is defined over $K$. We may hence write $W = \bigcup_{i=1}^{n} (x_i + B)$ for $x_i \in A^m(\ovl{K'})$. As the abelian subvariety $B$ descends to $K$ and and hence it follows that for a Zariski open $U \subseteq A$, we must have $f^{-1}(y) = \bigcup_{i=1}^{n} (x_{i,y} + B)$ for all $y \in U$, where $x_{i,y}$ are points of $V^m$ that can depend on $y$. Then there is a Zariski open of $V^m$ which is stabilized by $B$ and hence $V^m$ has $B$ as a stabilizer, which contradicts $V$ having finite stabilizer. 
\end{proof}

By Proposition \ref{Induction2}, the morphism $f: V^m \to A$ is generically finite. We now consider the map $\delta: V^m \to A$, given by $(x_1,\ldots,x_m) \mapsto x_1 + \cdots + x_{m-1} - [m-1]x_m$. This is the Faltings--Zhang map, used by Zhang \cite{Zha98} in his proof of the Bogomolov conjecture over number fields. A similar map previously appeared in Faltings' proof of the Mordell--Lang conjecture \cite{Fal91}. Then this map is still generically finite thanks to geometric non-degeneracy of $V$ (by \cite[Corollary 5.6]{DM25}), and the diagonal $\Delta$ is contracted. The following construction is due to GPT6 Astra. 

\begin{proposition} \label{ExceptionalDivisor1}
Let $Y = V^m$ and consider it as a generically finite cover of $A$ by $\delta: Y \to A$. Then there exists a projective variety $X$ over $K$ and positive integer $s$ such that there is a closed embedding $\iota: X \xhookrightarrow{} Y^s_A$, where 
$Y_A^s := \underbrace{Y \times_A Y \times_A \cdots \times_A Y}_{s \text{ times}},$ so that the induced map $f: X \to A$ is surjective, generically finite and has an exceptional divisor. Here, an effective divisor $D \subseteq X$ is exceptional for $f$ if $\codim f(D) \geq 2$. 
\end{proposition}

\begin{proof}
 Since $\delta: Y \to A$ is generically finite, we may view $K(Y)$ as a finite extension of $K(A)$. Let's first assume that $K(Y)/K(A)$ is a separable extension. Let $E$ be the Galois closure of $K(Y)$ over $K(A)$. Let $\sigma_i: K(Y) \xhookrightarrow{} E$ be the different conjugate embeddings for $i=1,\ldots,r$. This gives us an injection 
 $$\Spec E \xhookrightarrow{} \Spec R = \Spec \underbrace{K(Y) \otimes_{K(A)} \cdots \otimes_{K(A)} K(Y)}_{r \text{ times }}$$
corresponding to the surjection $R \to E$ given by $x_1 \otimes \cdots \otimes x_r \mapsto \prod_{i=1}^{r} \sigma_i(x_i)$. We now view $\Spec R$ as the generic fiber of the natural map $\delta: Y^r_A \to A$. We let $\eta$ be the point corresponding to $\Spec E$ and we set $X$ to be the Zariski closure of $\eta$ inside $Y^r_A$. 
\par 
The Galois group $G = \Gal(E/K(A))$ naturally acts on $Y^r_A$, by permuting the factors of the fiber product, i.e. $g \cdot (x_1,\ldots,x_r) = (x_{\pi_g^{-1}(1)},\ldots,x_{\pi_g^{-1}(r)})$, where $\pi_g$ is the permutation on $(\sigma_1,\ldots,\sigma_r)$ that the action of $g$ induces. This induces the usual action of $G$ on $\Spec E$. Since $X$ is the Zariski closure of $\eta$, it follows that $G$ acts on $X$ too. Observe that $\delta: X \to A$ is not a finite morphism as it factors through $X \to Y \to A$ for any of the projection maps $\pi_i: X \to Y$, and the map $X \to Y$ is surjective while the map $\delta: Y \to A$ is contracts a positive dimensional subvariety. 
\par 
We now have the map $\delta: X \to A$ so that $g \circ \delta = \delta$ for all $g \in G$. Let $W = X/G$ be the finite quotient of $X$ by $G$, which exists because $G$ is a finite group. Then we have an induced morphism $\tilde{\delta}:W \to A$ which is birational, as they both share the same function field. However it cannot be an isomorphism, as otherwise $\delta: X \to A$ is a finite morphism. We now apply \cite[Corollary 2.63]{KM08}. Since $A$ is smooth, it is $\bb{Q}$-factorial and hence the corollary applies, giving us an exceptional divisor in $X$ as desired by pulling back the exceptional divisor in $W.$
\par 
If $K(Y)/K(A)$ is inseparable, we have to modify the argument as if we take the normal closure $E$ of $K(Y)/K(A)$, then the invariant subfield of the Galois group $\Gal(E/K(A))$ for $E$ need not equal $K(A)$. We remedy this by letting $E/K(A)$ be the Galois closure of the maximal separable subextension $F$ of $K(Y)/K(A)$. We refer the readers to \cite[\href{https://stacks.math.columbia.edu/tag/09HD}{Tag 09HD}]{stacks-project} for basic facts about purely inseparable extensions. We let $\sigma_i: F \xhookrightarrow{} E$ be the different conjugate embeddings for $i=1,\ldots,r$. Since $K(Y)/F$ is purely inseparable, these extend to give $\sigma_i: K(Y) \xhookrightarrow{} \ovl{E}$. We set $E'$ to be the compositum of the $\sigma_i(K(Y))$'s. Then $E'/E$ is purely inseparable, the Galois group $G = \Gal(E/K(A))$ acts on $E'$ and permutes the conjugate embeddings $\sigma_i$'s, and the fixed field $(E')^G$ is purely inseparable over $K(A)$. 
\par 
We perform the same construction on $Y_A^r$ again, but this time with the point $\Spec E'$ instead of $\Spec E$. Taking the Zariski closure, we obtain a closed subvariety $X \subseteq Y_A^r$ such that $G$ acts on $X$. Taking the quotient gives us a morphism $\tilde{\delta}: X \to A$, which need not be birational as $K(X) = (E')^G$ which could be larger than $K(A)$. 
\par 
We now take the normalization $A'$ of $A$ in $K(X)$ and we replace $X$ with its normalization $X'$. Then $\tilde{\delta}:X' \to A$ extends to $\tilde{\delta}: X' \to A'$ by \cite[\href{https://stacks.math.columbia.edu/tag/035H}{Tag 035H}]{stacks-project} as $X'$ is normal. We have the following lemma regarding $\pi: A' \to A$. 

\begin{lemma} \label{Radical1}
The morphism $\pi: A' \to A$ is finite and a universal homeomorphism. Hence $A'$ is also $\bb{Q}$-factorial. 
\end{lemma}

\begin{proof}
 Since $A' \to A$ is the normalization of $A$ over a finite extension, it follows that $A' \to A$ is finite. To show its a universal homeomorphism, we let $U = \Spec C$ be an affine open of $A$. As $K(A')/K(A)$ is purely inseparable, there exists $e > 0$ such that $x^{p^e} \in K(A)$ for all $x \in K(A')$. Let $B$ be the integral closure of $C$ inside $K(A')$. Then for all $b \in B,$ $b^{p^e} \in K(A)$ and is also integral over $C$, implying that $b^{p^e} \in C$ as $A$ is normal. Thus $C \xhookrightarrow{} B$ satisfies the hypothesis of \cite[\href{https://stacks.math.columbia.edu/tag/0BRA}{Tag 0BRA}]{stacks-project} which implies that $\Spec B \to \Spec C$ is a universal homeomorphism. Hence $\pi: A' \to A$ is a universal homeomorphism. 
 \par
 Now let $D' \subseteq A'$ be a prime Weil divisor. Then $\pi(D')$ is also a prime Weil divisor on $A$. Since $A$ is smooth, $\pi(D')$ is Cartier. Now the pullback $\pi^* \pi(D')$ is also Cartier, but $\pi^* \pi(D') = e D'$ for some positive integer $e > 0$. Hence $D'$ is $\bb{Q}$-Cartier and thu $A'$ is $\bb{Q}$-factorial. 
\end{proof}

Applying \cite[Corollary 2.63]{KM08}, as $A'$ is still $\bb{Q}$-factorial, we obtain an exceptional divisor on $X'$. Mapping it down to $X$ gives us the exceptional divisor we want, finishing the proof.   
\end{proof}

Applying Proposition \ref{ExceptionalDivisor1}, we obtain an exceptional projective variety $X$ with a closed embedding into $Y^s_A$ and hence into $A^s$, such that $f: X \to A$ is generically finite and has an exceptional divisor $D$. Since the induced $f: X \to A$ is generically finite and surjective, so is the map $\pi_i: X \to Y$. Now let $\ovl{L}$ be the canonical adelic line bundle on $A$. Recall that $Y = V^m$ which when viewed as a subset of $A^m$, there is an adelic line bundle $\ovl{L}'$ given by pulling back and tensoring the different $\ovl{L}$'s. 
\par 
We let $\ovl{M} = \pi_1^* \ovl{L}' + \cdots + \pi_r^* \ovl{L}'$ and let $\ovl{N} = f^* \ovl{L}$. Then both $\ovl{M}$ and $\ovl{N}$ are nef, semipositive adelic line bundles, and are geometrically semiample and big. In fact, $\ovl{M}$ is geometrically ample since $\iota: X \xhookrightarrow{} Y^s_A$ is a closed immersion and so if $p_i: Y^r_A \xhookrightarrow{} Y^r \to Y$ are the projection maps, then 
$$\ovl{M} = \iota^*(p_1^* \ovl{L}' + \cdots + p_r^*\ovl{L}')$$
and each $\ovl{L}'$ is geometrically ample, giving us that $\ovl{M}$ is geometrically ample. 
\par 
The morphism $f: X \to A$ commutes with composing $\delta: V^m \to A$ with $\pi_i: X \to Y$.

\begin{proposition} \label{ExceptonalDivisorNef}
There exists $ a > 0$ such that both $a \ovl{L}' - \delta^*\ovl{L}$ and $a \ovl{M} - \ovl{N}$ is nef on $Y$ and $X$ respectively. 
\end{proposition}

\begin{proof}
It suffices to find $a > 0$ such that $ a \ovl{L}' - \delta^* \ovl{L}$ is nef on $Y = V^m$. Let $p_i : V^m \to V$ be the different coordinate projections. Then $\ovl{L}' = \sum_{i=1}^{m} p_i^*\ovl{L}$. Now by the theorem of the cube \cite[Theorem 8.6.11]{BG06}, the function $q(\alpha) = \alpha^*L$ for $\alpha \in \Hom(A^m,A)$ is quadratic if $L$ is symmetric. We hence obtain the identity 
\begin{equation} \label{eq: Nef1}
m(m-1) \sum_{i=1}^{m} p_i^*L - \delta^* L =  \sum_{1 \leq i < j \leq m-1} (p_i - p_j)^*L + \sum_{i=1}^{m-1} ((m-1)p_i + p_m)^*L,
\end{equation}
coming from the general identity for quadratic forms
$$\left(\sum_{i} c_i^2 \right) \sum_i q(x_i) - q(\sum c_i x_i) = \sum_{i < j} q(c_j x_i - c_i x_j),$$
where we take $c_1 = \cdots = c_{m-1} = 1$ and $c_m = -(m-1)$. As $\ovl{L}$ has the canonical metric, it is compatible with the theorem of the cube and hence (\ref{eq: Nef1}) holds with $\ovl{L}$ replacing $L$, proving that $m(m-1) \ovl{L}' - \delta^* \ovl{L}$ is nef as desired.
\end{proof}

\begin{proposition} \label{ExceptionalDivisor2}
We have $\ovl{M} \cdot \ovl{N}^{g} = 0$ where $g = \dim A$. Hence $\ovl{M}^{g+1} = \ovl{N}^{g+1} = 0$.  
\end{proposition}

\begin{proof}
Since $a \ovl{L}' - \delta^* \ovl{L}$ is nef and $(\ovl{L}')^{g+1} = 0$, it follows that all mixed intersections between $\ovl{L}'$ and $\delta^*\ovl{L}$ vanish. Hence we have $\ovl{L}' (\delta^* \ovl{L})^g = 0$. By the projection formula, we have $\pi_i^*\ovl{L}' (f^* \ovl{L})^g = 0$ for each $i=1,\ldots,r$. Hence $\ovl{M} \cdot \ovl{N}^g = 0$. 
\par 
By the projection formula, we know that $\ovl{N}^{g+1} = 0$. We now apply Proposition \ref{HodgeIndex2}. As $\ovl{M}$ and $\ovl{N}$ are nef, we obtain $\ovl{M}^{g+1} = 0$ as desired.
\end{proof}

We now apply Proposition \ref{HeightEquidistribution2}. Let $D$ be an exceptional divisor for $f: X \to A$. It follows that 
$$h_{\ovl{M}}(D) = h_{\ovl{N}}(D).$$
However, $\ovl{N} = f^*\ovl{L}$. As $\codim f(D) \geq 2$, the projection formula implies that $h_{\ovl{N}}(D) = 0$. Thus $h_{\ovl{M}}(D) = 0$. Using our closed embedding $\iota: X \xhookrightarrow{} Y_A^s \xhookrightarrow{} Y^s$, we obtain a subvariety $D$ of $Y^s = V^{ms}$ with dimension $g-1$ with zero canonical height. 
\par
We now apply the same argument as in Proposition \ref{Induction2} to arrive at a contradiction through the following proposition.  
\begin{proposition} \label{Induction3}
Let $K$ be a countable GVF and $A$ an abelian variety over $K$ which satisfies $\Tr_{\ovl{K}/\ovl{k_0}}(A_{\ovl{K}}) = 0$ where $k_0$ is the height zero subfield of $K$. Assume that $V \subseteq A$ is a height zero subvariety which is not a coset. Furthermore assume that $(\dim V, \dim A)$ is minimal among all counterexamples. Let $D \subseteq V^{r}$ be a positive dimensional subvariety of height zero. Then $\gcd(\dim D, \dim V) > 1$. 
\end{proposition}

\begin{proof}
We proceed by induction on $\dim D$. First, observe that the image of $D$ to any $V$ must be either a point or the entire $V$. Indeed otherwise the image of $D$ is a positive dimensional subvariety of $V$ with canonical height zero, just as in the proof of Proposition \ref{Induction1}. By minimality of $V$, the image of $D$ must be a coset, say $x+B$. By Proposition \ref{Induction2}, we know that $\dim A \geq 2 \dim V$. This implies that 
$$\dim(A/B) = \dim A - \dim B \geq 2 \dim V - \dim B > 2 \dim V - \dim V = \dim V.$$
In particular, we cannot have $q(V) = A/B$ which contradicts Proposition \ref{Induction1}. 
\par 
Now let's assume that $\pi: D \to V$ is surjective where $\pi$ is the first projection. We let $\nu: V' \to V$ be the normalization of $V$ inside $K(V)$ and endow $K(V)$ with the polarized GVF structure $(K(V),  \nu^*\ovl{L},\ldots,\nu^*\ovl{L})$. As all prime divisors of $V$ have canonical height $> 0$ by assumption of minimality, it follows by the proof of Lemma \ref{HeightZeroSubfield1} that $A_{K(V)}$ remains traceless over the height zero subfield. 
\par 
Now as in the proof of Proposition \ref{Induction2}, it follows that $D_{K(V)}$ has height zero and lies inside $V_{K(V)}^{r-1}$. But $V_{K(V)}$ and $A_{K(V)}$ is still a minimal counterexample in terms of $(\dim V, \dim A)$ whereas $\dim D_{K(V)} = \dim D - \dim V$. Hence either $\dim D - \dim V = 0,$ or we can apply induction. Either way, we find that 
$$\gcd(\dim D - \dim V ,\dim V) > 1 \implies \gcd(\dim D,\dim V) > 1$$
as desired.
\end{proof}

\begin{proof}[Proof of Theorem \ref{GVFBogomolov1}]
Combining $\dim D = g-1$ combined with Proposition \ref{Induction3} gives us a contradiction, since we know from Proposition \ref{Induction2} that $\dim V$ divides $g = \dim A.$ Thus the Bogomolov for traceless GVFs must be true, proving Theorem \ref{GVFBogomolov1}.

\end{proof}

It is now straightforward to conclude the Bogomolov conjecture over GVFs where $A$ may not have trivial trace.

\begin{theorem} \label{GVFBogomolov2}
Let $K$ be a countable non-archimedean GVF and let $k$ be the height zero subfield. Assume both $K$ and $k$ are algebraically closed. Let $A/K$ be an abelian variety and let $B = \Tr_{K/k}(A)$ be its trace with the trace morphism $\Tr: B_K \to A$. Then if $V \subseteq A$ has canonical height zero, $V$ must be of the form $x + B + \Tr(W_K)$ where $x$ is a point of height zero, $B$ is an abelian subvariety, and $W \subseteq B$ a subvariety defined over $k_0$.  
\end{theorem}

\begin{proof}
By \cite[Lemma 5.0.3]{Hul26} (and noting that images of cosets are still images of cosets), it suffices to prove the conjecture for any isogenous abelian variety. There exists an abelian variety $C/K$ with $\Tr_{K/k}(C) = 0$ and an isogeny $B \times C \to A$. Hence it suffices to prove it for $B \times C$. As we know the Bogomolov conjecture for $C$ and $B$ is defined over $k$, \cite[Lemma 5.4.6]{Hul26} shows that the Bogomolov conjecture for $B \times C$ follows as desired.    
\end{proof}

\section{Uniform Height Gaps for Subvarieties in Positive Characteristic}
We now explain how to prove uniform height gaps for subvarieties in positive characteristic by using Theorem \ref{GVFBogomolov2}. Already in \cite[\S 4.2]{Hul26}, Hultberg showed that Theorem \ref{GVFBogomolov2} implies the new gap principle for subvarieties over $\ovl{\bb{Q}}$, and also over function fields of characteristic zero (see also \cite{MY25}). We show that Theorem \ref{GVFBogomolov2} also implies a new gap principle for subvarieties in positive characteristic. The main difficulty in the positive characteristic situation is that an abelian variety $A$ can be non-isotrivial but isogeneous to an isotrivial abelian variety $B$. To handle this, we will use the theory of $p$-divisible groups. We begin by reviewing some results on $p$-divisible groups by Lau--Nicole--Vasiu \cite{LNV13}. 
\par 
Let $k$ be an algebraically closed field of characteristic $p > 0$ and let $D$ be a $p$-divisible group over $k$. The isomorphism number $n_D$ is the smallest $n$ such that $D[p^n]$ determines $D$, i.e. if $E$ is another $p$-divisible group with $E[p^n] \simeq D[p^n]$, then $E \simeq D$. The distance $q_{D,E}$ between two $p$-divisible groups $D$ and $E$ is the minimal non-negative integer $m$ for which there exists an isogeny $\pi: D \to E$ with kernel killed by $p^m$. We say that $D$ is a minimal $p$-divisible group if $n_D = 1$ and we define the minimal height of $D$ to be $q_{D,D_0}$ where $D_0$ is the unique minimal $p$-divisible group in the isogeny class of $D$. 
\par 
Let $K = \ovl{k(C)}$ be the algebraic closure of a function field of a curve $C$ over an algebraically closed field $k$ of characteristic $p$ and let $A/K$ be an abelian variety. Then $D_A = A[p^{\infty}]$ defines for us a $p$-divisible group. We want to understand the minimal degree of an isogeny $\pi: A \to B$ such that $D_B$ is minimal. 

\begin{proposition} \label{IsogenyBound1}
There exists an isogeny $\pi: A \to B$, with $\ker(\pi) \subseteq A[p^{2g}]$, so that $D_B$ is minimal. 
\end{proposition}

\begin{proof}
Let $D_A = A[p^{\infty}]$. By \cite[Theorem 1.5]{LNV13}, we know that $q_{D_A} \leq \lfloor \nu(c) \rfloor$, where $\nu: [0,c+d] \to \bb{R}$ is the Newton polygon for $D_A$, with $d$ denoting the dimension and $c$ the codimension, where $c+d$ adds up to the height of $D_A$ which is $2g$. As noted by the discussion after \cite[Theorem 1.2]{LNV13}, we have 
$$\nu(c) \leq \frac{cd}{c+d} \leq \frac{(c+d)^2}{(c+d)} = 2g.$$
Hence we can find $\pi: D_A \to D_0$, with $D_0$ minimal and $\ker(\pi) \subseteq D_A[p^{2g}]$. Let $H = \ker(\pi)$. Then taking $B = A/\ker(\pi)$ gives us an isogenous abelian variety with $D_B \simeq D_0$ which is minimal as desired.
\end{proof}

If $A/K$ is an abelian variety with $D_A$ minimal, then if $A$ is isogenous to an abelian variety over $k$, it follows that $A$ itself is also defined over $k$. This resolves the issue noted at the beginning of the section: a non-isotrivial abelian variety $A$ may be isogenous to an isotrivial abelian variety $B$.

\begin{proposition} \label{MinimalTrace1}
Let $A/K$ be an abelian variety with $D_A$ minimal, and assume there exists $B/k$ so that there is an isogeny $\pi: B_K \to A$. Then $A$ itself is defined over $k$. 
\end{proposition}

\begin{proof}
Let $H = \ker(\pi)$, which is a finite flat group subscheme over $K$. We want to prove that $H$ descends to $H'$ over $k$, in which then $A \simeq (B/H')_K$ and hence $A$ is constant. We split $H$ to its $p$-part and $p$-primary part, giving us a decomposition $H = H_p \times H_0$, where $H_p$ is a finite flat $p$-group scheme and $H_0$ is finite and etale. Then $H_0$ automatically descends to $k$ and by quotienting out $H_0$, we may assume that $\pi$ is a $p$-primary isogeny.
\par 
We first note that any minimal $p$-divisible group $D$ over $K$, must have already been defined over $\ovl{\bb{F}}_p$. Indeed, choose a $p$-divisible group $D'$ over $\ovl{\bb{F}}_p$ with the same Newton polygon as $D$. After base changing to $k$, the group $D'_K$ lies in the same isogeny class as $D$. But both $D'_K$ and $D$ are minimal \cite[Corollary 4.9]{LNV13} and hence must be isomorphic to each other.     
\par 
We now choose $D'_A$ over $k$ so that $D'_{A,K} \simeq D_A$. Let $D_B$ be the $p$-divisible group of $B$ over $k$. The isogeny $\pi: B_K \to A$ gives us an isogeny $\pi_D: D'_{B,K} \to D'_{A,K}$. By \cite[Lemma 2.8]{LNV13}, 
$$\Hom(D_B,D'_A) \simeq \Hom(D_{B,K}, D'_{A,K}),$$
the isogeny $\pi$ descends to $\pi'_D: D_B \to D'_A$ and hence $\ker(\pi_D)$ also descends to $\ker(\pi'_D)$ which is defined over $k$. But $\ker(\pi) = \ker(\pi_D)$ and so $\ker(\pi)$ descends over $k$ as desired.
\end{proof}

We also have to understand the locus of abelian varieties whose $p$-divisible groups are minimal. Let $\mathcal{D}$ be a $p$-divisible group over a $\bb{F}_p$-scheme $S$. We have the following result.

\begin{theorem} \label{pDivisibleStrata1}
Assume that $\mathcal{D}$ has constant Newton polygon. Then the locus of $s \in S$ where the geometric fiber $\mathcal{D}_{\ol{s}}$ is minimal is closed in $S$. 
\end{theorem}

\begin{proof}
This follows from \cite[Theorem 1.1]{LNV13}.     
\end{proof}

Now consider the moduli space $\mathcal{A}_{g,n}$ over $\ovl{\bb{F}}_p$. For each possible Newton polygon, there is a locally closed subscheme $\mathcal{A}_{g,n}^{\nu}$ of $\mathcal{A}_{g,n}$, consisting of points whose $p$-divisible group has Newton polygon $\nu$ by \cite[Lemma 4.8]{LNV13}. Then the minimal locus $\mathcal{A}_{g,n}^{\nu,\min}$ is closed inside $\mathcal{A}_{g, n}^{\nu}$ by \cite[Theorem 1.1]{LNV13} and hence is locally closed inside $\mathcal{A}_{g,n}$. Note that if $K$ is any field and $A/K$ an abelian variety that maps to $\mathcal{A}_{g,n}^{\nu,\min}$, then the $p$-divisible group of $A$ is minimal.  
\par 
We are now ready to begin the proof of our new gap principle in positive characteristic. 

\begin{theorem} \label{NewGapPrinciple1}
Let $K = k(B)$ be a function field for a curve of characteristic $p$. Let $A$ be an abelian variety over $K$ and $L$ a symmetric ample line bundle. Let $V \subseteq A_{\ovl{K}}$ be a subvariety such that $\Stab(V)$ is finite and $V-V$ generates $A$. Then there exists $c = c(\dim A, \deg_L(V)) > 0$ and $M = M(\dim A, \deg_L(V))$, such that 
$$\h_{A,L}(V) \geq c h_{\Fal}(A_0)$$
where $A_0$ is an abelian variety with an isogeny $\pi: A \to A_0$ of degree at most $M$.
\end{theorem}

\begin{proof}
Let $\deg_L(V) = d$. Then as $V-V$ generates $A$, it follows from \cite[Lemma 2.5]{GGK21} that $\deg_L(A)$ is bounded in terms of $d$. Applying \cite[Lemma 2.2(iv)]{GGK21} allows us to assume that $(A,L)$ is a principal polarization by passing to an isogeneous abelian variety (whose degree is bounded by the lemma).
\par Now, fix some $n \geq 3$ not divisible by $p$. Consider the Hilbert scheme of the universal abelian variety $\mathfrak{A}/\mathcal{A}_{g,n}.$ Since $\deg_L(V)$ is bounded, we may restrict to a finite type scheme $\mathbb{H}$ for which all such $(A,V)$'s appear as a point on $\mathbb{H}/\mathcal{A}_{g,n}$. We further pass to a Zariski open $\mathbb{H}^{\circ}$ that parametrizes geometrically integral subvarieties $V$, with $\Stab(V)$ finite and $V-V$ generating $A$, and we give $\mathbb{H}^{\circ}$ the reduced scheme structure. 
\par Now let $T \subseteq \mathbb{H}^{\circ}$ be an irreducible component. Then we have a principally polarized abelian scheme $\mathcal{A}$ over $T$ and a flat family of geometrically integral subvarieties $\mathcal{V} \subseteq \mathcal{A}$. We may find a symmetric line bundle $\mathcal{L}$ on $\mathcal{A}$ inducing twice the polarization and then endow $\mathcal{L}$ with a canonical adelic metric $\ovl{\mathcal{L}}$. Since $\mathcal{V}$ is integral, we can form the Deligne pairing $$\ovl{M} = \langle \ovl{\mathcal{L}},\ldots, \ovl{\mathcal{L}} \rangle_{\mathcal{V}/\bb{H}^{\circ}}.$$ Then for any GVF field $K$, a GVF-valued point $s \in \bb{H}^{\circ}(K)$ corresponding to the pair $(A, V)$ satisfies $$h_{\ovl{M}}(s) = a \h_{A,L}(V)$$ for some fixed constant $a > 0$ depending only on $\deg_L(V).$ 
\par Now assume for the sake of contradiction that our theorem is false. By Proposition \ref{IsogenyBound1}, we may assume $A$ has minimal $p$-divisible group by taking an isogeny of degree dividing $p^{2g}$ from the original $A$ to an abelian variety $A'$ with minimal $p$-divisible group and replacing $V$ with its image in $A'$. 
\par Hence we may find a sequence of abelian varieties $A_n/K_n$, with minimal $p$-divisible group, and subvarieties $V_n$ with $(A_n,V_n)$ lying in $T$, such that 
 $$\h_{A_n,L}(V_n) < \frac{1}{n} h_{\Fal}(A_n).$$
 Since there are only finitely many Newton polygons, we may assume that $A_m \in \mathcal{A}_{g,n}^{\nu,\min}(K_n)$ for the same $\nu$. Now cover $\mathcal{A}_{g,n}$ with finitely many affine opens $U_i$. We may assume all the $A_m$'s lie in the same open, say $U_1$. Using $\mathcal{L}^4$, we may embed $\mathfrak{A}_{U_i} \xhookrightarrow{} \bb{P}^N$. Now by \cite[Theorem 1.1]{MY25}, combining with Zhang's fundamental inequality, there exists some $C = C(\dim A, \deg_L(V))> 0$ such that 
 \begin{equation} \label{eq: Embedding1}
 \h(\iota(V_n)) \leq C h_{\Fal}(A_n)
 \end{equation}
 where $\h(\iota(V))$ denotes the Weil height of $\iota(V)$ inside $\bb{P}^N$. 
 \par 
 Fix an ample line bundle $\ovl{M}$ on $T$. We degenerate using $\eps_n = h_{\ovl{M}}(x_n   )$, where we may assume that $h_{\ovl{M}}(x_n) > 0$ where $x_n = [(A_n,V_n)]$, as otherwise $h_{\Fal}(A_n) = 0$. Applying Proposition \ref{GVFConstruction1} to the sequence $([A_n,V_n],K_n[\eps_n])$, we obtain an ultrafilter limit $x_{\omega} = [A_{\omega},V_{\omega}]$ over $K_{\omega}$ that lies in $T(K_{\omega})$. Then by definition, we have $h_{\ovl{M}}(x_{\omega}) = 1 > 0$.  
 \par 
 We first claim that $h_{\Fal}(A_{\omega}) > 0$. Indeed let's assume that $h_{\Fal}(A_{\omega}) = 0$. Then by (\ref{eq: Embedding1}), we also have $\h(\iota(V_{\omega})) = 0$. Let $k_{\omega}$ be the height zero subfield of $K_{\omega}$. Then $\h(\iota(V_{\omega})) = 0$ and $h_{\Fal}(A_{\omega}) = 0$ implies that both $A_{\omega}$ and $V_{\omega}$ are defined over $\ovl{k_{\omega}}$. Hence $x_{\omega} \in T(\ovl{k_{\omega}})$, but this contradicts $h_{\ovl{M}}(x_{\omega}) > 0$. Hence $h_{\Fal}(A_{\omega}) > 0$.
 \par 
 Now as each $A_n$ lie in $\mathcal{A}_{g,3}^{\nu,\min}$, it follows that $A_{\omega}$ also does. Now we have $\h_{A_{\omega},L}(V_{\omega}) = 0$. Using Theorem \ref{GVFBogomolov2}, it follows that 
$$V_{\omega} = a + B_{\omega} + \Tr(W_{K_{\omega}})$$
where $B_{\omega}$ is an abelian subvariety of $A_{\omega}$, and $W$ is a subvariety of $\Tr_{\ovl{K_{\omega}}/\ovl{k_{\omega}}}(A_{\omega})$ where $k_{\omega}$ is the height zero subfield of $K_{\omega}$.  As $\Stab(V_{\omega})$ is finite, $B$ must be zero. Then as $V-V$ generates $A$, it follows that the image of $W_{K_{\omega}} - W_{K_{\omega}}$ generates $A$ and hence we must have $\Tr_{\ovl{K_{\omega}}/\ovl{k_{\omega}}}(A_{\omega})$ being full dimensional.
\par 
We now apply Proposition \ref{MinimalTrace1}. As $A_{\omega}$ has minimal $p$-divisible subgroup, it follows that $A_{\omega}$ is defined over $\ovl{k_{\omega}}$. Then $h_{\Fal}(A_{\omega}) = 0$, which is a contradiction. This finishes our proof, where we obtain uniform constants that may depend on the characteristic of $k$.
\par 
Finally for sufficiently large positive characteristic, as already noted by Hultberg \cite[Theorem 4]{Hul26}, the Bogomolov conjecture for GVFs in characteristic zero gives a uniform height gap for all sufficiently large positive characteristic. This removes the dependence on the characteristic $p$. 
\end{proof}

\printbibliography

\end{document}